\documentclass[11pt,letterpaper]{amsart}

\usepackage{amsmath,amssymb,amsthm}
\usepackage{verbatim}

\usepackage{xurl}
\usepackage[psdextra]{hyperref}
\usepackage[capitalize]{cleveref}
\IfFormatAtLeastTF{2026-06-01}{%
  \newcommand*{\preamblesharedtheorem}[2]{%
    \newtheorem{#1}[theorem]{#2}%
  }%
}{%
  \usepackage{aliascnt}
  \newcommand*{\preamblesharedtheorem}[2]{%
    \newaliascnt{#1}{theorem}%
    \newtheorem{#1}[#1]{#2}%
    \aliascntresetthe{#1}%
  }%
}

\usepackage{enumerate}
\usepackage{enumitem}
\usepackage{physics}
\usepackage{todonotes}
\definecolor{crimson}{HTML}{C8102E}
\definecolor{stopgreen}{HTML}{00A14B}
\presetkeys{todonotes}{inline,color=yellow}{}

\usepackage[normalem]{ulem}

\usepackage[OT2,T1]{fontenc}
\usepackage{microtype}
\DeclareSymbolFont{cyrletters}{OT2}{wncyr}{m}{n}
\DeclareMathSymbol{\sha}{\mathalpha}{cyrletters}{"58}

\usepackage{tikz}
\usetikzlibrary{cd,decorations.markings}

\usepackage{graphicx}

\usepackage{outlines}
\usepackage{stmaryrd}
\usepackage{mathrsfs}
\usepackage{expl3}

\makeatletter
\DeclareFontFamily{OMX}{MnSymbolE}{}
\DeclareSymbolFont{MnLargeSymbols}{OMX}{MnSymbolE}{m}{n}
\SetSymbolFont{MnLargeSymbols}{bold}{OMX}{MnSymbolE}{b}{n}
\DeclareFontShape{OMX}{MnSymbolE}{m}{n}{
    <-6>  MnSymbolE5
   <6-7>  MnSymbolE6
   <7-8>  MnSymbolE7
   <8-9>  MnSymbolE8
   <9-10> MnSymbolE9
  <10-12> MnSymbolE10
  <12->   MnSymbolE12
}{}
\DeclareFontShape{OMX}{MnSymbolE}{b}{n}{
    <-6>  MnSymbolE-Bold5
   <6-7>  MnSymbolE-Bold6
   <7-8>  MnSymbolE-Bold7
   <8-9>  MnSymbolE-Bold8
   <9-10> MnSymbolE-Bold9
  <10-12> MnSymbolE-Bold10
  <12->   MnSymbolE-Bold12
}{}

\let\llangle\@undefined
\let\rrangle\@undefined
\DeclareMathDelimiter{\llangle}{\mathopen}{MnLargeSymbols}{'164}{MnLargeSymbols}{'164}
\DeclareMathDelimiter{\rrangle}{\mathclose}{MnLargeSymbols}{'171}{MnLargeSymbols}{'171}
\makeatother

\renewcommand{\phi}{\varphi}

\newcommand{\ovl}{\overline}

\providecommand{\ul}{\underline}
\renewcommand{\tilde}{\widetilde}
\renewcommand{\hat}{\widehat}

\ExplSyntaxOn
\cs_new_protected:Npn \preamble_new_operator:nn #1#2
  { \cs_new:cpn {#1} { \operatorname{#2} } }
\cs_new_protected:Npn \preamble_set_operator:nn #1#2
  { \cs_set:cpn {#1} { \operatorname{#2} } }
\cs_new_protected:Npn \preamble_new_font_alias:nnn #1#2#3
  { \cs_new:cpn {#1#2} { #3 {#2} } }

\clist_map_inline:nn
  {
    ord,Gal,Spec,Hom,SL,SO,GL,Isom,Ind,Cl,Pic,End,Sym,Sel,Tor,codim,Br,Th,Lind,Id,id,
    PSL,Aut,PGL,pr,Rad,Proj,depth,Ext,gr,nil,red,rad,Vect,MT,Hdg,Sch,Frob,Ver,Fr,fr,
    Vol,Div,Supp,cl,Sw,Coim,Gr,Ker,Coker,coker,Cov,Mod,Lie,Stab,acl,dcl,Mult,eq,ACF,
    Kr,Ded,Av,PD,st,Ch,Var,Swan,val,rk,Der,SU,SH,Ann,Cone,sel,tp,stp,supp,tor,gp,
    disc,Cn,Jac,Vis,WC,Comp,Frac,Func,Perv,Sp,diag,Nrd,an,Hilb,Cliff,Ad,Diff,sing,
    Out,Bl,rec,Inn,Loc,Coh,Aff,colim,coim,Ob,Sub,Def,alg,nr,un,ad,Fl,ab,fl,Ric,loc,
    sm,Mor,aut,Hol,An,Arg,BGL,BU,BO,full,om,Sing,St,Lk
  }
  { \preamble_new_operator:nn {#1} {#1} }

\preamble_new_operator:nn {Char} {char}
\preamble_new_operator:nn {Maxspec} {MaxSpec}
\preamble_new_operator:nn {ext} {Ext}
\preamble_new_operator:nn {height} {ht}
\preamble_new_operator:nn {vol} {Vol}
\preamble_new_operator:nn {rM} {r-M}
\preamble_new_operator:nn {im} {im}
\preamble_new_operator:nn {Sh} {SH}
\preamble_new_operator:nn {cone} {Cone}
\preamble_new_operator:nn {defect} {def}
\preamble_new_operator:nn {resid} {res}
\preamble_new_operator:nn {Graff} {GrAff}

\preamble_set_operator:nn {Form} {Form}
\preamble_set_operator:nn {div} {div}
\preamble_set_operator:nn {ker} {Ker}
\preamble_set_operator:nn {Im} {Im}
\preamble_set_operator:nn {df}{def}
\ExplSyntaxOff

\newcommand{\SHom}{\mathcal{H}\kern -.5pt\mathit{om}}
\newcommand{\shom}{\mathcal{H}\kern -.5pt\mathit{om}}
\newcommand{\sext}{\mathcal{E}\kern -.5pt\mathit{xt}}

\newcommand{\inj}{\hookrightarrow}

\newcommand{\xr}{\xrightarrow}

\renewcommand{\subset}{\subseteq}
\renewcommand{\supset}{\supseteq}

\ExplSyntaxOn
\clist_map_inline:nn {A,B,C,D,E,F,G,H,I,J,K,L,M,N,O,P,Q,R,S,T,U,V,W,X,Y,Z}
  {
    \preamble_new_font_alias:nnn {b}  {#1} {\mathbf}
    \preamble_new_font_alias:nnn {bb} {#1} {\mathbb}
    \preamble_new_font_alias:nnn {s}  {#1} {\mathscr}
    \preamble_new_font_alias:nnn {mc} {#1} {\mathcal}
  }
\clist_map_inline:nn {a,b,c,d,e,f,g,h,i,j,k,l,m,n,o,p,q,r,s,t,u,v,w,x,y,z}
  { \preamble_new_font_alias:nnn {mf} {#1} {\mathfrak} }
\ExplSyntaxOff

\catcode`,\active

\catcode`\,12

\newtheorem{theorem}{Theorem}[section]

\crefname{theorem}{Theorem}{Theorems}
\Crefname{theorem}{Theorem}{Theorems}
\crefname{thm}{Theorem}{Theorems}
\Crefname{thm}{Theorem}{Theorems}

\preamblesharedtheorem{proposition}{Proposition}
\crefname{proposition}{Proposition}{Propositions}
\Crefname{proposition}{Proposition}{Propositions}

\preamblesharedtheorem{prop}{Proposition}
\crefname{prop}{Proposition}{Propositions}
\Crefname{prop}{Proposition}{Propositions}

\preamblesharedtheorem{lemma}{Lemma}
\crefname{lemma}{Lemma}{Lemmas}
\Crefname{lemma}{Lemma}{Lemmas}

\preamblesharedtheorem{lem}{Lemma}
\crefname{lem}{Lemma}{Lemmas}
\Crefname{lem}{Lemma}{Lemmas}

\newtheorem*{lemma*}{Lemma}

\preamblesharedtheorem{claim}{Claim}
\crefname{claim}{Claim}{Claims}
\Crefname{claim}{Claim}{Claims}

\preamblesharedtheorem{corollary}{Corollary}
\crefname{corollary}{Corollary}{Corollaries}
\Crefname{corollary}{Corollary}{Corollaries}

\preamblesharedtheorem{conjecture}{Conjecture}
\crefname{conjecture}{Conjecture}{Conjectures}
\Crefname{conjecture}{Conjecture}{Conjectures}

\newtheorem{innerimportantthm}{Theorem}
\newenvironment{importantthm}[1]
  {\renewcommand\theinnerimportantthm{#1}\begin{innerimportantthm}}
  {\end{innerimportantthm}}
\crefname{innerimportantthm}{Theorem}{Theorems}
\Crefname{innerimportantthm}{Theorem}{Theorems}

\newtheorem{innerimportantcor}{Corollary}

\crefname{innerimportantcor}{Corollary}{Corollaries}
\Crefname{innerimportantcor}{Corollary}{Corollaries}

\theoremstyle{definition}
\preamblesharedtheorem{definition}{Definition}
\crefname{definition}{Definition}{Definitions}
\Crefname{definition}{Definition}{Definitions}

\preamblesharedtheorem{nondefinition}{Non-Definition}
\crefname{nondefinition}{Non-Definition}{Non-Definitions}
\Crefname{nondefinition}{Non-Definition}{Non-Definitions}

\preamblesharedtheorem{exercise}{Exercise}
\crefname{exercise}{Exercise}{Exercises}
\Crefname{exercise}{Exercise}{Exercises}

\preamblesharedtheorem{example}{Example}
\crefname{example}{Example}{Examples}
\Crefname{example}{Example}{Examples}

\preamblesharedtheorem{remark}{Remark}
\crefname{remark}{Remark}{Remarks}
\Crefname{remark}{Remark}{Remarks}

\preamblesharedtheorem{observation}{Observation}
\crefname{observation}{Observation}{Observations}
\Crefname{observation}{Observation}{Observations}

\preamblesharedtheorem{question}{Question}
\crefname{question}{Question}{Questions}
\Crefname{question}{Question}{Questions}

\preamblesharedtheorem{problem}{Problem}
\crefname{problem}{Problem}{Problems}
\Crefname{problem}{Problem}{Problems}

\numberwithin{equation}{section}

\usepackage[style=alphabetic]{biblatex}
\hypersetup{
  pdflang            = {en-US},
  pdftitle           = {Complex curves in o-minimal geometry},
  pdfauthor          = {Spencer Dembner},
  pdfsubject         = {Mathematics},
  pdfkeywords        = {o-minimality, definable complex analysis, coherent cohomology, Riemann surfaces, definable compactification},
  pdfdisplaydoctitle = true,
}

\begin{document}

\pagestyle{plain}

\title{Complex curves in o-minimal geometry}
\author{Spencer Dembner}
\date{September 30, 2026}

\begin{abstract}
    There has recently been considerable progress relating o-minimality to complex analytic geometry. Yet almost nothing is known about coherent cohomology or the classification of vector bundles, even for curves. In $\mathbb{R}_{\mathrm{an}}$ and similar structures, we show that cohomology of noncompact curves is concentrated entirely at punctures. As an application, we compute the cohomology of the structure sheaf on a punctured Riemann surface, and show that it is governed by a Diophantine approximation question. Finally, we use similar techniques to characterize which definable Riemann surfaces have definable compactifications. The proofs are based on a careful analysis of boundary behavior for definable holomorphic functions.
\end{abstract}

\maketitle
\setcounter{tocdepth}{1}
\tableofcontents

\section{Introduction}

Let $S$ be an o-minimal structure on $\bbR$. Peterzil and Starchenko \cite{PeterzilStarchenkoTameComplexAnalysis} developed a theory of \emph{definable} complex analysis, by replacing holomorphic functions with $S$-definable ones. Building on their work, Bakker-Brunebarbe-Tsimerman \cite{BakkerBrunebarbeTsimerman} introduced definable analytic spaces and definable coherent sheaves, proved definable analogues of Oka's coherence theorem and of GAGA, and used them to resolve a longstanding conjecture of Griffiths about the images of period maps in Hodge theory.

It is tempting to define the \emph{cohomology} of a definable sheaf by simply taking derived functors of the definable global sections functor. But very little is known about definable coherent sheaf cohomology, and it appears completely pathological: In the usual structure $\bbR_{\an, \exp}$ even the unit disk $\bbD$ has nontrivial cohomology, and in \emph{any} o-minimal structure the punctured disk $\bbD^*$ does. This diverges sharply from ordinary complex geometry, since coherent cohomology on noncompact Riemann surfaces always vanishes.

Even so, in this paper we show that in many o-minimal structures, the cohomology of curves is well-behaved. Call $S$ \emph{branching} if, roughly speaking, all singularities of one-variable definable holomorphic functions are branch points. We will see later that $\bbR_{\an}$, $\bbR_{\an}^K$, and $\bbR_{\an^*}$, among others, are examples of branching structures.

\begin{importantthm}{A}\label{thm:Vanishing}
    Let $S$ be a branching o-minimal structure that contains $\bbR_{\an}$, let $Y$ be an irreducible compact complex curve, and let $X \subsetneq Y$ be a noncompact definable open subset whose normalization has no punctures. Then:
    \begin{enumerate}
        \item Every definable coherent sheaf on $X$ has vanishing higher cohomology.
        \item Every definable vector bundle on $X$ is definably trivial.
    \end{enumerate}
\end{importantthm}

\begin{corollary}\label{cor:CartanA}
    With the same hypotheses, every definable coherent sheaf on $X$ is generated by global sections.
\end{corollary}

\begin{remark}
    As we will see in \cref{sec:Comments}, the hypotheses of \cref{thm:Vanishing} are sharp: If the structure is not branching, or does not contain $\bbR_{\an}$, or if $X$ has punctures, then one cannot expect cohomology to vanish.
\end{remark}

Using \cref{thm:Vanishing}, we can compute the cohomology of the definable structure sheaf on a punctured Riemann surface $X$, when the structure is $S = \bbR_{\an}^K$. Recall that for $K \subset \bbR$ a subfield, $\bbR_{\an}^K$ is the extension of $\bbR_{\an}$ generated by the power functions $x^r$ for $r \in K$. It turns out that the size of $H^1(X, \mcO)$ depends on a Diophantine approximation property of $K$, which we state precisely in \cref{sec:StructureSheaf}.

\begin{importantthm}{B}\label{thm:computation}
    Let $S = \bbR_{\an}^K$, for some subfield $K \subset \bbR$. Let $X \subset Y$ be a noncompact definable open subset of the compact Riemann surface $Y$, with punctures $p_1, \ldots, p_n$.
    \begin{enumerate}
        \item There is an injective comparison map
        \[ \eta \colon (\bbC\{z\}[z^{-1}])^n \to H^1(X, \mcO), \]
        where the left-hand side represents convergent Laurent series.
        \item If $K \subset L \cap \bbR$, where $L$ is one of the following, then $\eta$ is an isomorphism:
        \begin{itemize}
            \item The field $\ovl{\bbQ}$ of algebraic numbers.
            \item Any field $\bbQ(e^{\alpha_1}, \ldots, e^{\alpha_m})$, where the $\alpha_i$ are algebraic.
            \item The field $\bbQ(\pi)$.
        \end{itemize}
        \item Let $K = \bbQ(t_1, \ldots, t_m)$ for some $t \in \bbR^m$. The set of $t$ for which $\eta$ is an isomorphism has full measure.
        \item Nonetheless, there exist fields $K = \bbQ(t)$ for which $\eta$ is not an isomorphism on $\bbA^1$.
    \end{enumerate}
\end{importantthm}

For example, the theorem applies to the affine line as a once-punctured subset of $\bbP^1$.

One could justifiably wonder how general these theorems are: Concretely, given a definable Riemann surface $X$, does it definably embed as an open subset of a compact Riemann surface $Y$? For ordinary Riemann surfaces, the answer is essentially yes: By a theorem of Stout \cite[Theorem~8.1]{StoutRiemannSurfaces}, every finite-type Riemann surface embeds into a compact one, and we can even require its complement to be a union of points and disks.

In the definable setting, things are more complicated. We will see that there exist ``twisted forms'' of the punctured disk $\bbD^*$, whose puncture cannot be definably filled due to a monodromy obstruction. There is also a subtler obstruction, coming from the fact that definable holomorphic functions (usually) extend continuously to the boundary: There could be an end of $X^{\an}$ which analytically is a disk, but which definably ``appears'' as a puncture, and such an end cannot be definably compactified.

To measure this obstruction, we construct a definable space $\hat{X}$ called the \emph{intrinsic closure} of $X$. The space $\hat{X}$ is compact Hausdorff, contains $X$ as a dense open subset, and heuristically represents how the closure $\ovl X \subset Y$ ``should'' look, from the point of view of the definable structure on $X$. Thus, we call an isolated point of $\hat{X} \setminus X$ a \emph{definable puncture}. Our final theorem shows that definable punctures form the only obstacle to definable compactification.

\begin{importantthm}{C}\label{thm:compactification}
    Suppose $S$ is branching and contains $\bbR_{\an}$, and let $X$ be a definable Riemann surface. If $X$ has no definable punctures, then $X$ definably embeds into a compact Riemann surface $Y$. We can require that $\partial X \subset Y$ be a finite union of simple closed curves, and that $X$ is connected at every boundary point.
\end{importantthm}

If $X \subset Y$ is already embedded, then definable punctures and punctures are the same, and it follows that there exists a ``tame'' definable embedding.

\begin{corollary}\label{cor:TameEmbeddability}
    Suppose that $S$ is branching and contains $\bbR_{\an}$. Let $Y$ be a compact Riemann surface, and let $U \subset Y$ be any definable open subset. Then there exists a definable embedding $U \inj Y'$, with $Y'$ another compact Riemann surface, such that $\partial U$ is a union of points and simple closed curves, and such that $U$ is connected at every boundary point.
\end{corollary}

\begin{remark}
    Can we require, as in the analytic setting, that the complement of $U$ be a union of points and disks? In general we cannot, because the definable Riemann mapping theorem is false: Most definable simply connected domains are not definably biholomorphic to the disk. We currently have no characterization of the domains which \emph{are} definably biholomorphic to the disk, though Kaiser \cite{KaiserRiemannMapping} has proven some results about a related problem.
\end{remark}

\subsection{Previous work}

Definable coherent sheaves and their pathologies have recently appeared in work of Esnault and Kerz \cite{EsnaultKerzGrauert}, who show that the analogue of Grauert's direct image theorem is false. The cohomology of definable coherent sheaves is also the subject of forthcoming work by Brosnan and Melrod, whose results are disjoint from the ones presented here.

\subsection{Outline}

In \cref{sec:DefinableTopology}, we outline the proof of \cref{thm:Vanishing}, review definable sheaf cohomology, and discuss some properties of coverings which will be needed in the sequel. In \cref{sec:BoundaryHolomorphic}, we study how definable holomorphic functions behave near boundary points, and introduce the class of branching structures. \cref{sec:CohomologyVanishing} largely completes the proof of \cref{thm:Vanishing}. In \cref{sec:Comments}, we prove \cref{cor:CartanA}, discuss other implications of \cref{thm:Vanishing}, and demonstrate that its hypotheses are sharp. In \cref{sec:StructureSheaf}, we prove \cref{thm:computation}. Finally, in \cref{sec:Compactification}, we define the intrinsic closure, deduce its most important properties and prove \cref{thm:compactification}.

\subsection{Acknowledgments}

We thank Ravi Vakil, Jacob Tsimerman, Daniel Kim, Sergei Starchenko, Ben Church, Zoe Morgan, Adam Melrod, Patrick Brosnan, Benny Zak, and Gal Binyamini for a number of interesting and useful conversations. While working on this paper the author was supported by an NSF Graduate Research Fellowship, under grant DGE-1656518. Finally, Claude was used to produce and edit the figures.

\subsection*{Conventions}

From \cref{sec:CohomologyVanishing} onward we adopt the following convention: \emph{All geometric objects are definable by default.} This includes spaces, open subsets, holomorphic functions, sheaves, covers, sections, cocycles and cohomology classes. In particular, sheaves $\mcF$ on $X$ are always assumed definable, and $H^i(X, \mcF)$ always denotes definable sheaf cohomology. The exceptions will be explicitly labeled as ``ordinary'' or ``analytic'' to avoid confusion.

\section{Some definable topology}\label{sec:DefinableTopology}

Here is an impressionistic outline for the proof of \cref{thm:Vanishing}:

\begin{enumerate}
    \item Cohomology classes are represented by \v{C}ech cocycles. Fix a cocycle $s$ with transition functions $s_I$.
    \item For a fixed boundary point $x \in \partial X$, use properties of the structure $S$ to show that the $s_I$ extend globally near $x$. It follows that $s$ is trivial near $x$.
    \item Therefore, there exists a cover $(V_i)$ of $\partial X$ such that $s$ is trivial on $(V_i \cap X)$. We will see that $s$ extends holomorphically past almost all points of $\partial X \cap V_i$.
    \item We can refine $(V_i)$ to eliminate the finitely many singular points, making the $s_I$ overconvergent. Thus, $s$ extends past the boundary.
    \item Since $S$ contains $\bbR_{\an}$, any trivialization of $s$ on a larger domain restricts definably to $X$.
\end{enumerate}

In this section, we isolate the purely topological aspects of this argument, since they also play a role in \cref{sec:Compactification}. First, we review the definition and basic properties of o-minimal sheaf cohomology.

\subsection{Definable sheaf cohomology}

Given a definable space $X$, we would like to define sheaves of definable continuous functions on $X$. In general, definable functions do \emph{not} satisfy gluing for the Euclidean topology: For example, in $\bbR_{\an}$, every analytic function is ``locally'' definable for the Euclidean topology, but not globally. On the other hand, suppose the cover $(U_i)$ is a finite cover by definable open sets. Then definable functions \emph{do} satisfy the sheaf condition, simply because a finite union of definable sets is definable. In this section we explain this notion of a definable sheaf, the corresponding notion of sheaf cohomology, and some of its basic properties.

O-minimal sheaves were first considered by Pillay \cite{PillaySheaves}, and o-minimal sheaf cohomology was defined by Edmundo, Jones, and Peatfield \cite{EdmundoJonesPeatfield}. Their results are stated only for definable sets $X \subset \bbR^n$, and a priori this is not quite enough for us: We want to consider for example Riemann surfaces given by charts. Thus we have the following definition of van den Dries.

\begin{definition}[{\cite[Chapter~10]{vandendriestame}}]
    A \emph{definable topological space} is an ordinary topological space $X$ with a finite definable atlas: There exist finitely many definable sets $X_i \subset \bbR^{N_i}$, definable open subsets $X_{ij}$, and definable homeomorphisms
    \[X_{ij} \to X_{ji}  \]
    satisfying a cocycle condition.
\end{definition}

We will only need to consider definable spaces $X$ which are locally compact and Hausdorff, and by \cite[Theorem~10.1.8]{vandendriestame}, every such $X$ definably embeds as a subset of $\bbR^n$. Therefore, the results of \cite{EdmundoJonesPeatfield} hold in our setting as well.

\begin{definition}
    The definable site $\ul X$ of $X$ is the site whose objects are definable open subsets of $X$, whose morphisms are inclusions, and whose coverings are generated by finite definable open covers.
\end{definition}

\begin{definition}
    A definable sheaf $\mcF$ on $X$ is a sheaf on the definable site $\ul X$.
\end{definition}

Though definable sheaves a priori represent sheaves on a site, they can equivalently be viewed as sheaves on a certain topological space.

\begin{proposition}[{\cite[Prop.~3.2]{EdmundoJonesPeatfield}}]
    There is a functor
    \[X \mapsto \tilde{X}  \]
    from definable spaces to ordinary spaces, called the \emph{o-minimal spectrum}, with the following properties:
    \begin{enumerate}
        \item Given definable open subsets $U_i$ of $X$, $(U_i)$ is a definable cover of $X$ if and only if $(\tilde{U}_i)$ is a cover of $\tilde{X}$.
        \item Definable sheaves on $X$ are equivalent to sheaves on $\tilde{X}$.
    \end{enumerate}
\end{proposition}

\begin{definition}
    Let $\mcF$ be a definable (abelian) sheaf on $X$. The cohomology groups $H^i(X, \mcF)$ are given by the derived functors of the global sections functor.
\end{definition}

Note that the definition of $H^i(X, \mcF)$ depends only on the category of abelian sheaves. Thus, we have
\[ H^i(X, \mcF) = H^i(\tilde{X}, \mcF).  \]

\begin{remark}
    We adopt the following definability conventions: If $X$ is a definable space, then a sheaf on $X$ is implicitly a definable sheaf, and $H^i(X, \mcF)$ always denotes definable sheaf cohomology unless we specify otherwise. Sections, cocycles, and cohomology classes for definable sheaves are assumed to be definable unless stated otherwise. A ``cover of $X$'' is implicitly assumed to be finite and definable.
\end{remark}

Following an observation of Carral and Coste \cite[Corollary~1]{CarralCoste}, one can study sheaves on $\tilde{X}$ by restricting to the subspace of closed points, which is a compact Hausdorff space. Thus, definable sheaf cohomology inherits many favorable properties of sheaf cohomology on well-behaved spaces.

\begin{proposition}[{\cite[Proposition~4.1--4.2]{EdmundoJonesPeatfield}}]\label{prop:EdmundoJonesPeatfield}
    Let $\mcF$ be a sheaf on $X$.
    \begin{enumerate}
        \item  If $i > \dim(X)$, then we have $H^i(X, \mcF) = 0$.
        \item One can compute $H^i(X, \mcF)$ by taking \v{C}ech cohomology for the definable site.
    \end{enumerate}
\end{proposition}

One also has the usual Mayer--Vietoris sequence.

\begin{proposition}\label{prop:MayerVietoris}
    Let $X = U \cup V$, where $U, V$ are definable open sets, and let $\mcF$ be an abelian sheaf on $X$. We have a long exact sequence
    \[
    \begin{aligned}
    \cdots \to H^n(X, \mcF)
    &\to H^n(U, \mcF) \oplus H^n(V, \mcF) \\
    &\to H^n(U \cap V, \mcF)
    \to H^{n+1}(X, \mcF) \to \cdots .
    \end{aligned}
    \]
\end{proposition}

\begin{proof}
    The Mayer--Vietoris sequence holds for sheaf cohomology on any topological space, though not any site. Since definable sheaf cohomology equates to sheaf cohomology on $\tilde{X}$, we conclude. \qedhere
\end{proof}

\subsection{Boundary covers}

Let $Z$ be a definable compact Hausdorff space of dimension 2; later we will take $Z$ to be a compact complex curve $Y$, or the intrinsic closure $\hat{X}$ defined in \cref{sec:Compactification}. Let $X \subset Z$ be a definable open subset. We are interested in properties which hold ``locally along the boundary'' --- for example, that a cocycle is locally trivial near every boundary point. In particular, we are interested in extending a property such as triviality that holds near the boundary, \emph{across} the boundary. In this section we prove a series of abstract covering lemmas in this direction.

Our crucial tool is the o-minimal triangulation theorem of van den Dries.

\begin{theorem}[{\cite[Theorem~8.2.9]{vandendriestame}}]\label{thm:Triangulation}
    Let $B_1, \ldots, B_n$ be definable subsets of $Z$. Then there is a definable triangulation of $Z$ compatible with the $B_i$, in the sense that each is a union of (open) simplices.
\end{theorem}

\begin{definition}
    When $K$ is a triangulated space, we will use the following notation:
\begin{enumerate}
    \item $K_n$ is the set of $n$-simplices in $K$.
    \item For a simplex $\sigma$, the \emph{open star} $\St_K(\sigma)$ is the union of interiors over all simplices $\tau$ containing $\sigma$. It is a triangulated open subset of $K$. If $X$ is any triangulated subset of $K$, we let
    \[ \St_X(\sigma) = \St_K(\sigma) \cap X. \]
\end{enumerate}
\end{definition}

\begin{definition}
    \leavevmode\par
    \begin{enumerate}
        \item If $x \in \partial X$, then a \emph{boundary neighborhood} of $x$ is an intersection $U \cap X$, where $U$ is a definable neighborhood of $x$ in $Z$.
        \item A \emph{boundary collar} $C \subset X$ is an intersection $V \cap X$, where $V$ is a definable neighborhood of $\partial X$ in $Z$.
        \item Let $(V_i)$ be a cover for $\partial X$ in $Z$. Then its restriction $(V_i \cap X)$ is a \emph{boundary cover}.
    \end{enumerate}
\end{definition}

\begin{lemma}\label{lem:NoTripleIntersect}
    \leavevmode\par
    \begin{enumerate}
        \item If $(U_i)$ is a boundary cover of $X$, then it has a refinement whose closures $\ovl{U_i}$ have no triple intersections. Indeed, if $U_i = V_i \cap X$ for a cover $(V_i)$ of $\partial X$, then we can assume that $(\ovl{V_i})$ has no triple intersections.
        \item If $(U_i)$ is any cover of $X$, then it has a refinement with no triple intersections near $\partial X$.
    \end{enumerate}
\end{lemma}

\begin{figure}[ht]
  \centering
  \includegraphics[width=\linewidth,
    alt={Two panels, each with the boundary of X drawn as a horizontal line. Left panel: three overlapping half-disks U i, U j, U k in a row along the boundary line, each completed below the line by a dotted arc; consecutive sets overlap, but no three share a common point. Right panel: three wedges U i, U j, U k fanning out from a single boundary point p above the line, with only neighboring wedges overlapping.}]{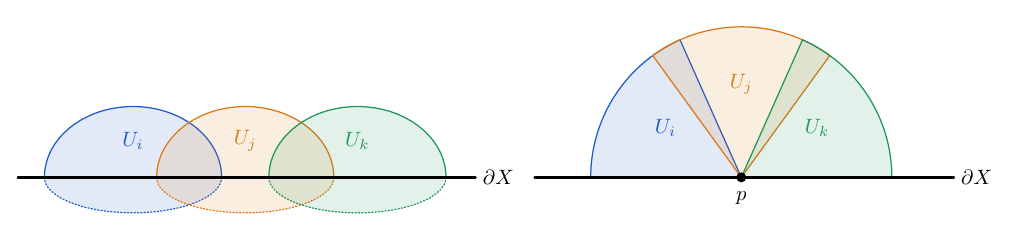}
  \caption{Two-ply refinement (1) of a boundary cover, and (2) of an arbitrary cover restricted near the boundary.}
  \label{fig:two-ply-covers}
\end{figure}

\begin{proof}
    By \cref{thm:Triangulation}, we reduce to proving the analogous facts when $Z$ is a finite simplicial complex of dimension 2, $X$ is a triangulated open subset, and $\partial X$ is a closed one-dimensional subcomplex. After subdividing, we can assume that there is no 2-simplex $\sigma$ whose vertex set is contained in $\partial X$.

    To prove fact (1), fix a cover $(V_i)$ extending $(U_i)$, and assume the triangulation is compatible with $(V_i)$. In this case the cover
    \[ \{\St_Z(x): x \in (\partial X)_0  \} \]
    is a refinement of $(V_i)$ covering $\partial X$. It has no triple intersections, by the subdivision condition above. By compactness of $\partial X$, after shrinking we can assume the closures have no triple intersections either.

    To prove (2), triangulate compatibly with the $(U_i)$, and consider
    \[ \{\St_Z(x): x \in X_0 \}. \]
    This is a cover of $X$ refining $(U_i)$. The triple intersection $\St_Z(x_1) \cap \St_Z(x_2) \cap \St_Z(x_3)$ contains at most the 2-simplex $\{x_1, x_2, x_3 \}$, and since the $x_i$ are interior points, no triple intersection meets the boundary. \qedhere
\end{proof}
\begin{lemma}\label{lem:FiniteOmissionsLemma}
    Let $(U_i)$ be a boundary cover of $X$, and for each $i \neq j$ let $E_{ij} \subset \partial X \cap \partial U_{ij}$ be a finite set. Then there exists a refinement $(U_i')$ with the following properties (all closures in $Z$):
    \begin{enumerate}
        \item We have $\ovl{U_{ij}'} \subset U_{ij} \cup \partial X$.
        \item The set $E_{ij}$ is disjoint from $\ovl{U_{ij}'}$.
    \end{enumerate}
\end{lemma}

\begin{figure}[ht]
  \centering
  \includegraphics[width=0.78\linewidth,
    alt={Two overlapping caps U i (blue) and U j (orange) rise from the horizontal boundary line of X. Their overlap U ij is outlined in red, and its curved edges, the interior boundary of U ij minus the boundary of X, are labeled. Two red points of the finite set E ij sit on the boundary line at the ends of the overlap. Dashed refined caps U i prime and U j prime sit inside the original caps; their overlap U ij prime has pulled away from the curved edges and from the points of E ij.}]{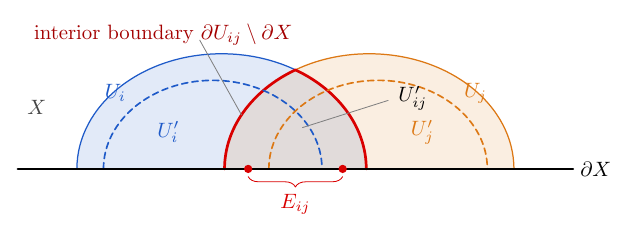}
  \caption{The proof of \cref{lem:FiniteOmissionsLemma}.}
  \label{fig:finite-omissions}
\end{figure}

In other words, by refining the cover we can shrink $U_{ij}$ away from all interior boundary points, and finitely many problematic points of $\partial X$.

Here and throughout the paper, it will often be useful to have the following shorthand: $A$ is \emph{compactly contained in} $B$, written $A \Subset B$, if $\ovl{A} \subset B$ is compact.

\begin{proof}
    Let $(U_i)$ be our boundary cover, and let $(V_i)$ be a cover of $\partial X$ extending it. By compactness of $\partial X$, we can replace each $V_i$ by some $V_i' \Subset V_i$ and obtain a refined cover. The restriction $V_i' \cap X$ then satisfies condition (1). For condition (2), pick some $x \in E_{ij}$. Since $(U_i)$ forms a boundary cover, choose some $V_k \ni x$. There exists a definable neighborhood $R_x \ni x$ with $\ovl{R_x} \subset V_k$. We simply define $U_a' = U_a \setminus \ovl{R_x}$, for all $a \neq k$. The resulting refinement is still a boundary cover, and now we have $x \notin \ovl{U_{ij}'}$. Continuing by induction, we conclude. \qedhere
\end{proof}

The next result will allow us to extend a cocycle past the boundary, given that we can extend its transition functions.

\begin{lemma}\label{lem:CoverExtension}
    Let $(U_i)$ be a boundary cover, and for each $i, j$ let $A_{ij}$ be a definable open neighborhood in $Z$ of $\ovl{U_{ij}}$. After refining $(U_i)$, it arises from a cover $(V_i)$ with $V_i \cap V_j \subset A_{ij}$.
\end{lemma}

\begin{figure}[ht]
  \centering
  \includegraphics[width=0.7\linewidth,
    alt={Two overlapping caps U i and U j above the horizontal boundary line of X, each completed below the line by a dotted arc to an open set V i or V j of the ambient space. A dashed green oval labeled A ij encloses the closure of the overlap of U i and U j and extends across the boundary line.}]{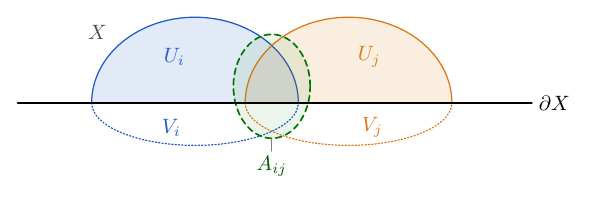}
  \caption{The proof of \cref{lem:CoverExtension}.}
  \label{fig:cover-extension}
\end{figure}

\begin{proof}
    Let $(V_i)$ be any cover extending $(U_i)$. Choose a triangulation of $Z$ compatible with everything including the $A_{ij}$, and subdivide so that the following holds: If $\{x, y\}$ is a simplex with $x, y \in \partial X$, then $\{x, y\} \in (\partial X)_1$. The sets
    \[ V_i' = \bigcup_{x \in (\partial X \cap V_i)_0} \St_Z(x) \]
    form a refined cover of $\partial X$. If a simplex $A$ is contained in $V_i' \cap V_j'$, then we have $A \supset \{ x, y \}$ for some $x \in \partial X \cap V_i, y \in \partial X \cap V_j$. It follows that $\{x, y \} \in (\partial X)_1$, so there is a simplex $A' \subset X$ that has the face $\{x, y \}$. We have $A' \in (U_i \cap U_j)_2$, so $\{x, y \} \in \ovl{U_{ij}}$. It follows that $A \in (A_{ij})_2$. \qedhere
\end{proof}

\subsection{Boundary-connectedness}

In this subsection, we briefly introduce an important technical hypothesis.

\begin{definition}
    The definable open subset $U \subset Z$ is \emph{boundary-connected} if every $z \in \partial U$ has a basis of neighborhoods $V$ such that $V \cap U$ is connected.
\end{definition}

The importance of this hypothesis is that it rules out ``spurious'' boundary behaviors. For example, as we will see in the next section, if $U$ is boundary-connected then every definable holomorphic $f \colon U \to \bbC$ extends continuously to its boundary. Otherwise, there are simple counterexamples: Take $U$ to be the complement of the real line and take the following function:
\[ f(z) = \begin{cases} 1 & \Im(z) > 0 \\ 0 & \Im(z) < 0 \end{cases} \]

\begin{figure}[ht]
  \centering
  \includegraphics[width=\linewidth,
    alt={Two examples of boundary-connectedness failing. Left: a disk U with a thin slit I removed along a radius; a small dashed neighborhood V of a point x on the slit meets U in two pieces, one on each side of the slit. Right: a region U shaped like an annulus pinched to a single point x on its outer edge; a small dashed neighborhood V of x meets U in two pieces.}]{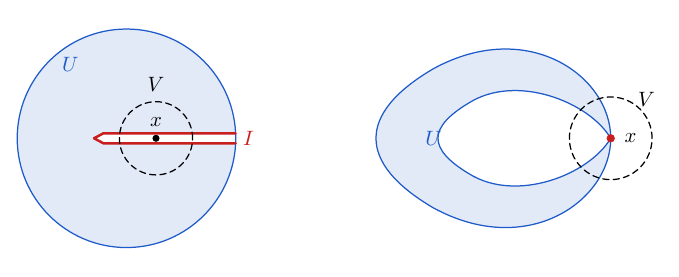}
  \caption{Failures of boundary-connectedness.}
  \label{fig:bc-failures}
\end{figure}

\begin{lemma}\label{lem:SimplicialCharacterizationofBoundaryConnectedness}
    Suppose $Z$ is triangulated compatibly with $U$. Then $U$ is boundary-connected if and only if, for all simplices $\sigma$ of $\partial U$, the ``punctured star''
    \[ \St_U(\sigma) \setminus \sigma \]
    is connected.
\end{lemma}

\begin{proof}
    First, suppose that $I \in (\partial U)_1$ is a 1-simplex. Near every $i \in I$, we can write $U$ as the disjoint union of 2-simplices $\sigma_i$ having $I$ as a face. Thus, $U$ is connected at $i$ if and only if there is a unique such face, which is precisely the given hypothesis. Second, suppose that $x \in (\partial U)_0$ is a vertex. The set $\St_U(x)$ is connected if and only if the link $\Lk(x,U)$ is connected, and this property is preserved by simplicial refinements. \qedhere
\end{proof}

\begin{proposition}\label{prop:BoundaryAgreement}
    Let $X \subset Z$ be definable, open and boundary-connected, and let $U \subset X$ be definable open. Near all but finitely many shared boundary points $x \in \partial U \cap \partial X$, we have $U = X$.
\end{proposition}

\begin{proof}
    Fix a triangulation of $Z$ compatible with $X, U$, and their boundaries. By \cref{lem:SimplicialCharacterizationofBoundaryConnectedness}, if $I \subset \partial X$ is a 1-simplex, then there is a unique 2-simplex $\sigma \in X_2$ with $I \subset \partial \sigma$. If $I \subset \partial U$, then there must also exist a 2-simplex $\tau \subset U$ with $I \subset \partial \tau$, and therefore $\sigma = \tau$. This shows that $U = X$ near every $x \in \partial X \cap \partial U$, except at vertices. \qedhere
\end{proof}

\begin{lemma}\label{lem:StarsAreBoundaryConnected}
    If $\sigma \in Z_n$ is a simplex, then the star $\St_Z(\sigma)$ is boundary-connected.
\end{lemma}

\begin{proof}
    By \cref{lem:SimplicialCharacterizationofBoundaryConnectedness} it's enough to show the following: If $\tau \in (\partial \St_Z(\sigma))_m$, then
    \[ \St_Z(\sigma) \cap \St_Z(\tau) \setminus \tau \]
    is connected. If the intersection is nonempty, then it follows that $\gamma = \tau \cup \sigma$ is a simplex (here we are identifying simplices with their vertex sets). We have $\gamma \neq \tau$, because $\tau$ is not contained in $\St_Z(\sigma)$. It follows that
    \[ \St_Z(\sigma) \cap \St_Z(\tau) \setminus \tau = \St_Z(\gamma), \]
    and the result is immediate since stars are connected. \qedhere
\end{proof}

Suppose now that $Z = Y$ is a definable compact Riemann surface (for example, if $S$ contains $\bbR_{\an}$, then every compact Riemann surface is definable). In this case we can be more explicit about boundary-connectedness. Note that if $U \subset \bbC$ is open, the definition of boundary-connectedness is relative to $\bbP^1$ since $\bbC$ is not compact.

\begin{lemma}\label{lem:ConnectednessAtZero}
    Let $U \subset Y$ be definable, open, and connected at $z \in \partial U$. Then either $U$ contains a punctured neighborhood, or there is a definable holomorphic chart taking $z$ to $0$ and taking $U$ to
    \[ (f,g) = \left\{z: \Arg(z) \in (f(|z|), g(|z|)) \right\}, \]
    where $f, g \colon (0, \varepsilon) \to S^1$ are definable continuous functions. If $U$ is boundary-connected, then $f \neq g$.
\end{lemma}

\begin{proof}
    Assume $U$ contains no punctured neighborhood. Restricting to a neighborhood of $z$, we can identify $z$ with $0$ and $U$ with an open subset of the punctured unit disk. The map
    \[ \bbD^* \to (0,1) \times S^1, z \mapsto (|z|, z / |z|) \]
    is definable. Taking cell decomposition in the new coordinates, we get a decomposition
    \[U = \bigcup_{i=1}^n (f_i, g_i), \text{with } (f_i, g_i) = \{(r, \theta): \theta \in (f_i(r), g_i(r))\} \]
    into disjoint radial intervals, none of which equals $\bbD^*$. Since $U$ is connected near $0$, we have $n = 1$. Finally, if $f = g$ then $U$ will not be boundary-connected at the points $(r, f(r))$. \qedhere
\end{proof}

\section{Boundary behavior of holomorphic functions}\label{sec:BoundaryHolomorphic}

Let $f \colon U \to \bbC$ be a definable holomorphic function of one variable. If $z \in \partial U$ is an isolated singularity, then $f$ is not essential at $z$. But at a general point $z \in \partial U$, its behavior can be more subtle: For example, $e^{1 / z}$ has an essential singularity and is $\bbR_{\an, \exp}$-definable on an open set near zero, though not on a punctured neighborhood. The proofs of our main theorems will require a detailed understanding of such phenomena, because they proceed essentially by reduction to the boundary.

In this section, we survey the boundary properties of definable holomorphic functions. As we will indicate, many of these results are due to Peterzil--Starchenko and Wilkie, while others appear to be original. We also introduce the class of branching structures which form the context for the rest of the paper, and prove some basic properties.

\subsection{Continuity at the boundary}

The first observation we will need is the following, which slightly adapts a theorem of Peterzil and Starchenko. It says definable holomorphic functions extend continuously to their boundary unless there is a topological obstruction.

\begin{proposition}[cf. {\cite[Theorem~2.57]{PeterzilStarchenkoExpansions}}]\label{prop:ctsext}
    Let $U \subset \bbP^1$ be a definable open set connected at $x \in \partial U$, and let $f \colon U \to Y$ be a definable holomorphic map to a definable compact Riemann surface (for example, $\bbP^1$). Then $f$ has a unique limiting value at $x$. If $U$ is boundary-connected, then there exists a definable continuous extension
    \[ \ovl f \colon \ovl{U} \to Y.  \]
\end{proposition}

\begin{remark}
    By o-minimality, a definable \emph{continuous} function $U \to Y'$ already extends to all but finitely many boundary points. The real content of the theorem is that when $f$ is holomorphic, its set of limiting values at $z \in \partial U$ is always finite.
\end{remark}

\begin{proof}
    It is enough to prove the first claim, so pick some $x \in \partial U$. By compactness there exists a limit point $(x,y)$ of $\Gamma(f)$. Let $V_y \subset Y$ be a definable open chart containing $y$, and let $(U_i)$ be a cover of $f^{-1}(V_y)$. Refining the cover, we may assume that each $U_i$ is connected at $x$. It follows by \cite[Theorem~2.57]{PeterzilStarchenkoExpansions} that
    \[ f \colon U_i \to V_y \]
    attains a unique limit point over $x$. If we let $\Gamma_x$ be the set of limiting values over $x$, then it follows that $\Gamma_x \cap V_y$ is finite, and hence that $\Gamma_x$ is discrete. By o-minimality, $\Gamma_x$ must be finite, and by \cite[Fact~2.2]{PeterzilStarchenkoExpansions}, the set $\Gamma_x$ is connected. Thus, $\Gamma_x$ is a singleton. \qedhere
\end{proof}

\subsection{Univalence}

Recall that a holomorphic function is \emph{univalent} if it is injective, or equivalently a biholomorphism onto its image. It will be important, particularly in \cref{sec:Compactification}, to understand which definable holomorphic functions are univalent. In this section we consider a definable holomorphic $f \colon U \to \bbC$, with $0 \in \partial U$, and give a criterion for $f$ to be univalent near zero. As in the previous section, we allow $S$ to be any o-minimal structure.

\begin{lemma}\label{lem:UnivalenceCriterion}
    Suppose that $U$ is a definable open set which is connected near $0 \in \partial U$, and let $f \colon U \to \bbC$ be a definable nonconstant holomorphic function with $\lim_{z \to 0} f(z) = 0$. Shrink so $f$ is nonvanishing, and define
    \[ g \colon U \to S^1, \quad g(z) = \frac{f(z)}{|f(z)|}. \]
    \begin{enumerate}
        \item For all sufficiently small $r$, the function $g_r(\lambda) = g(r \lambda)$ is strictly monotone.
        \item If $g_r$ is injective for small $r>0$, then $f$ is univalent near $0$.
        \item Fix some $N > 2 \pi$, and suppose that for all small $r >0$, the variation of $g_r$ is greater than $N$. Then $f(U)$ contains a punctured neighborhood near zero.
    \end{enumerate}
\end{lemma}

\begin{proof}
    Note that the intersection $U \cap \partial D(0,r)$ is connected for sufficiently small $r$, by \cref{lem:ConnectednessAtZero}.

    \textbf{(1)}: Let us define
    \[ Z = \left\{z: \Im\left(\frac{ i z f'(z)}{f(z)}\right) = 0 \right\}. \]
    It's enough to show that $Z$ does not accumulate at zero. Indeed, it follows for small $r$ that $g_r$ has no critical points and is strictly monotone. Suppose instead that there exists a definable curve germ $\gamma(t) \subset Z$ tending to $0$. We have
    \[ f(\gamma(x)) = \int_0^x f'(\gamma(t)) \gamma'(t) \, dt. \]
    As $t \to 0$, the arguments of $f'(\gamma(t))$ and $\gamma'(t)$ converge, without loss of generality both to zero. Thus, the argument of the right-hand side converges to zero. But since $\gamma \subset Z$, the argument of the left-hand side converges to $\pm \pi / 2$, and we have a contradiction.

    \textbf{(2)}: Since $f'$ has finitely many zeros, we may shrink to assume that $f'$ is nonvanishing, in which case $g'$ is nonvanishing as well. Under this assumption, the set $C_{\lambda} = g^{-1}(\lambda)$ is a finite union of smooth curves, each mapping injectively to $\lambda \bbR_{>0}$. To show injectivity, it suffices to show $C_{\lambda}$ is connected near zero. If not, $C_{\lambda}$ contains two curve arcs $C_1, C_2$, and any sufficiently small circle $\partial D(0,r)$ meets both $C_i$. But this contradicts injectivity of $g_r$.

    \textbf{(3)}: If $f(U)$ contains no punctured neighborhood, it omits a definable curve germ $C$ near zero. Since $C$ is definable, its argument converges near zero: Thus, we can choose some $\delta > 0$ and a branch of $\log$, defined on $D(0, \delta) \setminus \ovl{C}$, whose imaginary width is less than $N$. Since $\lim_{z \to 0} f(z) = 0$, we have
    \[ f^{-1}(D(0,\delta) \setminus \ovl{C}) \supset D(0, \varepsilon) \cap U, \]
    for some $\varepsilon >0$. It follows that $\log \circ f$ has imaginary width less than $N$ near the origin. But this contradicts our assumption on $g_r$.\qedhere
\end{proof}

\begin{corollary}\label{cor:CurveOmittingUnivalence}
    Let $U$ be a definable open set connected near $0 \in \partial U$, and let $f \colon U \to \bbC$ be a nonconstant definable holomorphic function. If $f(U)$ does not contain a punctured neighborhood of $f(0)$, then $f$ is univalent near zero.
\end{corollary}

\begin{figure}[ht]
  \centering
  \includegraphics[width=\linewidth,
    alt={Left: a curved wedge-shaped domain with a vertex at 0, bounded by two curves C lambda leaving 0 and an outer arc. A curved arrow labeled f, univalent, points to the right-hand picture: a disk centered at 0 with a branch cut drawn along the negative real radius, which is omitted from the image of f.}]{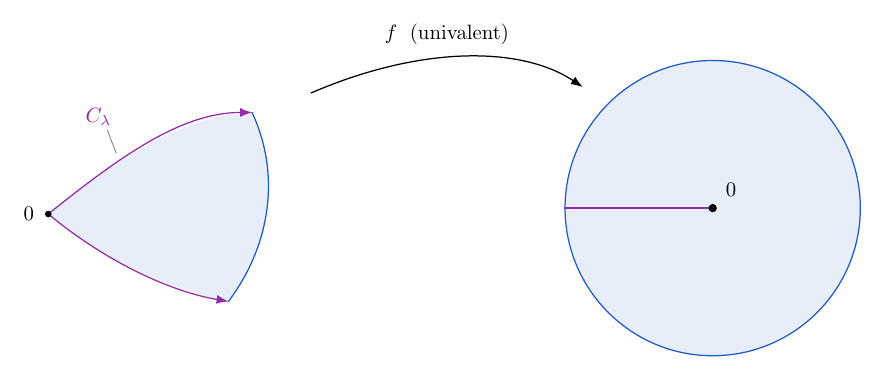}
  \caption{A definable holomorphic function $f$ whose image omits a branch cut is univalent near zero.}
  \label{fig:univalence}
\end{figure}

\begin{proof}
    Suppose without loss of generality that $f(0) = 0$ (here we mean the continuous limiting value, which exists by \cref{prop:ctsext}). As we saw while proving \cref{lem:UnivalenceCriterion}, it is enough to show that $C_{\lambda}$ is connected near zero. Suppose not: Then it contains two curve components $C_1, C_2$. Since $\lim_{z \to 0} f(z) = 0$, the image of each $C_i$ is a ray $(0, t)$ by connectedness.

    Choose a nontrivial circular arc
    \[ A = [a_1, a_2] \subset \partial D(0,r) \cap U, \]
    with endpoints $a_i \in C_i$, and assume without loss of generality that $|f(a_2)| \geq |f(a_1)|$. If $B \subset C_2$ is a segment whose image connects $f(a_2)$ to $f(a_1)$, then $f(A \cup B)$ is a loop based at $f(a_1)$, and part (1) of \cref{lem:UnivalenceCriterion} shows it has nontrivial monodromy around the origin. But since $f(U)$ contains no punctured neighborhood, it must omit a curve germ $C$, and the complement of $C$ is simply connected. \qedhere
\end{proof}

\begin{lemma}\label{lem:AffineArgument}
    Let $f \colon U \to \bbC$ be a nonconstant definable holomorphic function near zero, where $U$ is a radial sector. Define, for suitable $\lambda \in S^1$, the function
    \[ h(\lambda) = \lim_{r \to 0}  \frac{f(r \lambda)}{|f(r \lambda)|}. \]
    Then we have $h(\lambda) = \lambda^r \alpha$, for some $r \in \bbR$ and $\alpha \in S^1$.
\end{lemma}

\begin{proof}
    By o-minimality, the function $h$ is well-defined since limits along rays always exist. If we restrict so $f$ is nonvanishing and let $u = \Im(\log(f))$, then our goal equivalently is to show that
    \[ \lim_{z \to 0} (u(z) - \ell(z)) = 0, \]
    for some function $\ell(z) = r \Arg(z) + \Arg(\alpha)$. Observe that $\ell$ and $u$ are both harmonic, so $v:= u - \ell$ is as well. By o-minimality, the function $u$ is bounded near zero, or else $f / |f|$ would have unbounded variation.

    Let us write the boundary of $U$ as
    \[ \partial U = L_1 \cup L_2 \cup A, \]
    where $L_1, L_2$ are the boundary rays and $A$ is a radial arc along $\partial D(0, R)$. The limit of $u$ is well-defined along $L_1, L_2$. Thus, choosing $\ell$ appropriately, we can assume that $v(z)$ converges to zero along $L_1, L_2$. Let $\varepsilon > 0$. Then after shrinking $R$, we can assume $|v(z)| \leq \varepsilon$ along the $L_i$.

    Since $v$ is bounded and harmonic, for $z \in U$ we have
    $$v(z) = \int_{\partial U} v(x) d \omega^z(x), $$
    where $\omega^z(x)$ is the harmonic measure based at $z$. Since $0$ is a boundary point not in $\ovl{A}$, as $z \to 0$ the harmonic measure of $A$ converges to zero. We then have $\limsup_{z \to 0} |v(z)| \leq 2 \varepsilon$. Since $\varepsilon$ was arbitrary, we conclude. \qedhere
\end{proof}

Our univalence criterion has the following application, which may be of independent interest.

\begin{proposition}\label{prop:Univalence}
    Let $f \colon U \to \bbC$ be a nonconstant definable holomorphic function, where $U$ is a definable open neighborhood of $(0, \varepsilon)$.

    \begin{enumerate}
        \item There exists a definable open $V \subset U$, containing a smaller interval $(0, \delta)$, on which $f$ is univalent.
        \item If $U$ has positive angular width around the positive reals, then we may assume $V$ does as well.
    \end{enumerate}
\end{proposition}

\begin{proof}
    First we prove (1), following the notation of \cref{lem:UnivalenceCriterion}. We may assume without loss of generality that $f(z) \to 0$ as $z \to 0$. By o-minimality, $g(t)$ converges as $t \to 0$, without loss of generality to $1$. Therefore, if $V' \subset S^1$ is a small definable neighborhood of $1$, then $g^{-1}(V')$ contains some $(0, \delta)$, and we conclude by applying the lemma.

    Now, we prove (2). By \cref{lem:AffineArgument}, the function
    \[ h(\lambda) = \lim_{r \to 0} g(r \lambda) \]
    is continuous. Therefore, if $V'$ is a neighborhood as above, then $h^{-1}(V')$ is open. We conclude that $V$ has positive angular width near zero, as desired. \qedhere
\end{proof}

\subsection{Extension past the boundary}

We now shift gears to a different kind of question: To what extent can we extend $f$ holomorphically past the boundary to a larger domain? If
\[ f \colon \bbD \to \bbC \]
is holomorphic and not definable, then it could be singular at every boundary point; for example the function
\[ \sum_n z^{2^n} \]
exhibits this behavior. However, if $f$ is $S$-definable and the $S$-definable sets are piecewise analytic (analytic cell decomposition), then as Wilkie first observed, essential boundaries are impossible. Indeed, a weaker property than analytic cell decomposition suffices.

\begin{definition}
    The structure $S$ is \emph{unary-analytic} if every definable function $f \colon (a,b) \to \bbR$ is piecewise analytic, where $- \infty \leq a < b \leq \infty$.
\end{definition}

\begin{lemma}
    Let $S$ be unary-analytic. Then any 1-cell $C \subset \bbR^n$ is piecewise analytic, and any definable function $f \colon C \to \bbR$ is piecewise analytic.
\end{lemma}

\begin{proof}
    Work by induction on $n$. The case $n = 1$ is immediate from the definition, so suppose $n > 1$. By the definition of a cell, $C$ is equal to either $\{z \} \times (a,b)$, for $z \in \bbR^{n-1}$, or $\Gamma(g)$, where $C' \subset \bbR^{n-1}$ is a 1-cell and $g \colon C' \to \bbR$ is a definable continuous function. First suppose that $C = \{z \} \times (a,b)$. Then $C$ is trivially analytic, and $f$, which reduces to a function $(a,b) \to \bbR$, is analytic by definition.

    Otherwise, suppose that $C = \Gamma(g)$. By induction, $C'$ is analytic and $g$ is piecewise analytic, so $\Gamma(g)$ is clearly analytic. The function $f \circ (\id, g)$ is piecewise analytic by induction, and we have
    \[ f = (f \circ (\id, g)) \circ \pi, \]
    so we conclude. \qedhere
\end{proof}

The following is a slight generalization of \cite[Theorem~1]{WilkieMittagLeffler}.

\begin{proposition}\label{prop:GenericHolExtension}
    Assume $S$ is unary-analytic, and let $U \subset \bbP^1$ be a definable, boundary-connected open set. Then a definable holomorphic function \mbox{$f \colon U \to \bbC$} has at most finitely many boundary singularities.
\end{proposition}

\begin{proof}
    By \cref{prop:ctsext}, there exists a definable continuous extension $\ovl f \colon \ovl{U} \to \bbP^1$. Let $\tilde{f}$ denote the boundary value along $\partial U$: Since $S$ is unary-analytic, after covering $\bbP^1$ by charts we conclude that $\tilde{f}$ is real analytic, except at finitely many points. Thus, choose a boundary point $x \in \partial U$ where $\tilde{f}$ is analytic. Discarding finitely many points and changing coordinates, we may assume near $x$ that $\partial U$ is equal to a 1-cell of the form $C = \Gamma(g)$, where
    \[ g \colon (a,b) \to \bbR \]
    is a definable real-analytic function. Extending $g$ holomorphically and applying the local change of coordinates
    \[ z \mapsto z + i \tilde{g}(z), \]
    we reduce to the case where $C = (c,d)$ is simply a real interval (note that the change of coordinates is generally not definable). Observe that $\tilde{f}$ takes the value $\infty$ only finitely many times along $(c,d)$. Otherwise, the function $1 / \tilde{f}$ vanishes along an interval. By the Schwarz reflection principle, $1 / f$ extends across the interval, and by the identity theorem it must be identically zero, giving a contradiction.

    Suppose $x \in (c,d)$ is a point where $\tilde{f}(x) \neq \infty$. Then $\tilde{f}$ has a holomorphic extension near $x$, simply by power series expansion. If we let $h$ denote the extension, then the difference $f - h$ is holomorphic and converges to zero along $(c,d)$. Therefore, by Schwarz reflection and the identity theorem, we have $f = h$. \qedhere
\end{proof}

Note that in general, there may not be a single \emph{definable} extension of $f$ covering all nonsingular boundary points.

\subsection{Branching structures}

The unary-analyticity property controls behavior of functions at ``generic'' boundary points, but as we will see this is not enough for cohomology vanishing. In $\bbR_{\an, \exp}$, for example, a definable holomorphic
\[f \colon U \to \bbC,  \]
with $U$ as above, has finitely many boundary singularities, but as remarked above those singularities can be \emph{essential}. In this section we isolate a class of structures whose boundary singularities are relatively controlled.

\begin{definition}
    Let $S$ be an o-minimal structure on $\bbR$. We say that $S$ is \emph{branching} if every definable $f \colon (0, \varepsilon) \to \bbR$ extends to a definable holomorphic function on
    \[ D(0, \delta) \setminus (-\delta, 0 ], \]
    for some $\delta > 0$.
\end{definition}

In other words, all singularities are branch singularities, and definable continuation near zero is always possible after taking a branch cut.

\begin{proposition}
    Every branching structure is polynomially bounded. In other words, definable one-variable function germs grow at most polynomially at infinity.
\end{proposition}

\begin{proof}
    By the main theorem of \cite{MillerExponentiation}, a structure which is not polynomially bounded defines the real exponential function. It follows that $e^{1 / x}$ is definable near $0$. But its extension $e^{1 / z}$ is never definable on a branched neighborhood of zero; indeed, since the singularity is isolated, by taking closures we could define the function on a punctured neighborhood, and definable functions on punctured disks never have essential singularities. \qedhere
\end{proof}

\begin{proposition}
    Branching structures are unary-analytic.
\end{proposition}

\begin{proof}
    Let $f \colon (a, b) \to \bbR$ be definable. Since $S$ is branching, a suitable application of the definition shows that for all $x \in [a,b]$, including endpoints, $f$ is analytic on a punctured interval of the form
    \[ (x - \varepsilon, x) \cup (x, x+\varepsilon). \]
    If $a$ or $b$ is infinite, then the claim is that $f$ is analytic on a ray $(- \infty,u)$ or $(u, \infty)$, and it follows by working in the coordinate $1 / t$. By compactness of $[a,b] \subset [- \infty, \infty]$, we obtain a finite cover by such intervals, omitting only finitely many points. \qedhere
\end{proof}

\begin{figure}[ht]
  \centering
  \includegraphics[width=0.6\linewidth,
    alt={A domain V to the left of a curved boundary arc through the point 0, with a narrower sector U inside it ending at 0. A dashed green region around 0 is a small disk with a thin wedge removed on the right, outside V: the neighborhood onto which a function on U extends.}]{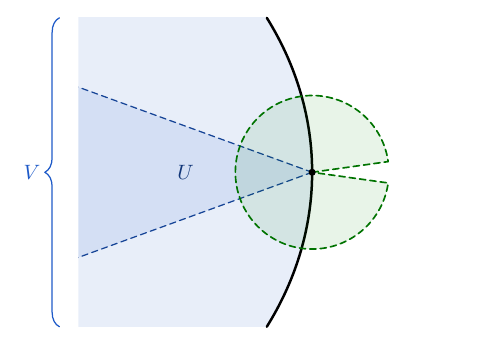}
  \caption{The proof of \cref{prop:BranchingExtension}.}
  \label{fig:extended-domain}
\end{figure}

\begin{proposition}\label{prop:BranchingExtension}
    Let $S$ be a branching structure, let $U \subset V \subset \bbC$ be definable open sets connected near $0$, and assume $V$ does not contain a punctured neighborhood of $0$. Then any definable holomorphic $f \colon U \to \bbC$ extends definably to some neighborhood
    \[ D(0, \delta) \cap V. \]
\end{proposition}

\begin{proof}
    Fix a definable curve germ $\gamma \colon (0, \varepsilon) \to U$, with $\lim_{t \to 0} \gamma(t) = 0$. Since $S$ is polynomially bounded, it follows from the proposition in \cite{MillerExponentiation} that we have
    \[ \lim_{t \to 0} \frac{\gamma(t)}{t^s} = \lambda \in \bbC \setminus \{0 \}, \]
    for some $s > 0$. The power function $t^s$ is definable, because
    \[ t^s = \lim_{x \to 0} \frac{\gamma(xt)}{\gamma(x)}. \]
    Since $S$ is branching, there exists a definable holomorphic function $\tilde{\gamma}$ extending $\gamma$ to a slit disk, and likewise the complex power function $z^s$ is definable on a slit disk. By \cref{prop:ctsext}, their ratio converges to $\lambda$. If we define $h(\theta)$ following \cref{lem:AffineArgument}, then we must have $h(\theta) = \theta^s \alpha$ for some $\alpha$.

    Since $V$ contains no punctured neighborhood, it omits a curve germ $C$ near the origin, and we let $A$ denote $D(0, \delta) \setminus \ovl{C}$. Replacing $\gamma$ by $\gamma(t^n)$, we can assume by \cref{lem:AffineArgument} that $\tilde{\gamma}$ has arbitrarily large angular width near zero and that its image contains a punctured neighborhood. Let $W$ be the component of $\tilde{\gamma}^{-1}(A)$ containing $(0, \varepsilon)$. Near zero, $W$ is bounded by $\tilde{\gamma}^{-1}(\ovl{C})$, which is a finite union of curve germs. It follows that the boundary $\partial (\tilde{\gamma}(W))$ is contained in $\ovl{C}$ near zero. After shrinking $\delta$, it follows that
    \[ \tilde{\gamma}(W) \supset D(0, \delta') \setminus \ovl{C} \]
    for some $\delta' > 0$; otherwise, the omitted points would accumulate along a curve germ other than $C$. On the other hand, since $A$ omits a curve germ, after shrinking $\delta$ the map $\tilde{\gamma}$ becomes univalent on $W$. Shrinking $W$ and $U$ appropriately, we can assume that
    \[ \tilde{\gamma} \colon W \to D(0, \delta) \setminus \ovl{C} \]
    is a biholomorphism.

    Now suppose that $f \colon U \to \bbC$ is definable holomorphic. The inverse image $\tilde{\gamma}^{-1}(U)$ is a neighborhood of $(0,t)$ for some $t >0$. Thus, $f \circ \tilde{\gamma}$ extends to a slit disk, and in particular to $W$. Precomposing with $\tilde{\gamma}^{-1}$, we see that $f$ extends to $D(0, \delta) \setminus \ovl{C}$, and thus extends near $0$ to $V$. \qedhere
\end{proof}

As we now show, there are many examples of branching structures.

\begin{proposition}\label{prop:BranchingExamples}
    The following o-minimal structures are branching: $\bbR_{\alg}$, $\bbR_{\an}$, more generally $\bbR_{\an}^K$ for a subfield $K \subset \bbR$, and the generalized power series structure $\bbR_{\an^*}$.
\end{proposition}

Since it will be needed in the proof, we review how $\bbR_{\an^*}$ is defined. However, the unfamiliar reader can follow the rest of the paper by simply assuming that $S = \bbR_{\an}^K$.

\begin{definition}[{\cite[p.~4377]{vandenDriesSpeisseggerGeneralizedPowerSeries}}]
    Suppose that $S_1, \ldots, S_m$ are well-ordered subsets of $[0, \infty)$ (in particular, this implies they are countable). The ring $\bbR[[X^*]]$ consists of generalized power series
    \[ F(X) = \sum_{\alpha} c_{\alpha} X^{\alpha}, \]
    whose support $\{\alpha: c_\alpha \neq 0 \}$ is contained in $S_1 \times \cdots \times S_m$. Given a polyradius $r = (r_1, \ldots, r_m)$, with $0 < r_i < \infty$, we define
    \[\norm{F}_r = \sum_\alpha |c_\alpha| r^{\alpha}. \]
    If $\norm{F}_r$ is finite, then $F$ defines a continuous function on $[0,r_1] \times \cdots \times [0,r_m]$ (nonnegative values are required to make $x^{\alpha}$ unambiguous).
\end{definition}

\begin{definition}
    The structure $\bbR_{\an^*}$ is the expansion of $\bbR_{\alg}$ by all overconvergent generalized power series. In other words, if $\norm{F}_r < \infty$, and $s_i < r_i$ for all $i$, then
    \[ F|_{[0, s_1] \times \cdots \times [0, s_m]}\]
    is definable.
\end{definition}

\begin{proof}[Proof of \Cref{prop:BranchingExamples}]
    In $\bbR_{\alg}$, every definable function germ is algebraic, and conversely every algebraic function is definable as long as its domain is definable. Thus, $\bbR_{\alg}$ is clearly branching.

    Next, consider $\bbR_{\an}^K$. If $f \colon (0, \varepsilon) \to \bbR$ is definable, then by \cite[Proposition~4.5]{MillerPowerFunctions}, there exists a convergent power series $F$, and $r_0, \ldots, r_k \in K$, with $r_i >0$ for $i \geq 1$, such that after shrinking $\varepsilon$ we have
    \[  f(x) = x^{r_0} F(x^{r_1}, \ldots, x^{r_k}). \]
    It follows by convergence that there is a holomorphic extension
    \[ g(z) = z^{r_0} F(z^{r_1}, \ldots, z^{r_k}) \]
    of $f$, valid on some branched neighborhood. Therefore, we only need to show that $g$ is $S$-definable. Since $S$ contains $\bbR_{\an}$, the analytic function $F$ is definable near zero, which means it is enough to show the power functions $z^{r_i}$ are definable.

    If $z$ lies in a branched neighborhood of $0$, then we can write
    \[ z = x+iy = r(z) e^{i \theta(z)}, \]
    where $r, \theta$ are analytic functions. The function $r = \sqrt{x^2 +y^2}$ is semialgebraic, and the function $\theta = \arctan(y / x)$ is $\bbR_{\an}$-definable. Thus, we can $S$-definably transform to polar coordinates. In polar coordinates, the power function has the form
    \[ (r, \theta) \mapsto (r^{r_i}, r_i \theta),  \]
    and since the real power function is definable it follows that $z^{r_i}$ is definable as well.

    We now handle $\bbR_{\an^*}$. By \cite[Theorem~B]{vandenDriesSpeisseggerGeneralizedPowerSeries}, every definable one-variable germ near zero has a convergent expression
    \[ f(x) = \sum_i a_i x^{\alpha_i}, \]
    where $\Sigma = (\alpha_i)$ is well-ordered. As above, we would like to show that the holomorphic expression
    \[ g(z) = \sum_i a_i z^{\alpha_i},  \]
    which converges, is definable. Equivalently, we want to show that its real and imaginary parts are. Writing $z = r e^{i \theta}$ as above, we have
    \[\Re(z^{\alpha_i}) = r^{\alpha_i} \Re(e^{ i \cdot \alpha_i \theta}) = r^{\alpha_i} \cos(\alpha_i \theta). \]
    If we let $\sum_n b_n z^n$ be the power series for cosine, then we have
    \[ \Re(g(z)) = \sum_i \sum_n (a_i b_n \alpha_i^n) r^{\alpha_i} \theta^n.  \]
    If we restrict to $\theta \geq 0$ or $\theta < 0$, this is a convergent generalized power series in the variables $r, \theta$. If $\Sigma$ is the well-ordered support of $f$, then the support of $\Re(g)$ is contained in $\Sigma \times \bbN$; it follows that $\Re(g)$ is definable near zero. The same argument applies to $\Im(g)$, and we conclude. \qedhere
\end{proof}

On the other hand, not every polynomially bounded structure is branching; for example the structure $\bbR_{\mcG}$ generated by multisummable series is not. Indeed, Theorem~12 of \cite{PadgettSpeisseggerDefinability} shows that there exist holomorphic functions on a sector near zero, which are definable on a strictly smaller sector near zero, but which do not extend even holomorphically to a branched neighborhood of the origin.

\begin{remark}
    In the first version of this paper we asked the following: If a function $f \colon (0,\varepsilon) \to \bbR$ is definable in a branching structure, does it necessarily possess a (convergent) generalized power series expansion at zero? Mirabi \cite{MirabiCanonicalExpansions} observed that this is not the case: The structure $\bbR_{\mcQ}$, first constructed by Kaiser, Rolin and Speissegger \cite{KaiserRolinSpeisseggerTransitionMaps}, is branching. This also gives another example of a structure for which our main theorems are true.
\end{remark}

\section{Cohomology vanishing}\label{sec:CohomologyVanishing}

In this section, we largely complete the proof of \cref{thm:Vanishing}; in the smooth case we will need to use \cref{cor:TameEmbeddability}, which is proved independently in the final section. In the first two subsections, we focus on solving ``definable Cousin problems,'' i.e., showing that $H^1(X, \mcO)$ and $H^1(X, \GL_n)$ are trivial. Then, in \cref{subsec:TheoremAProof}, we deduce cohomology vanishing for all coherent sheaves.

Initially, let $S$ be any o-minimal structure containing $\bbR_{\an}$, so that all compact complex curves are definable. From now on, we adopt the convention described in the introduction: \emph{all geometric objects are definable by default.}

\subsection{Overconvergence}

Let $Y$ be a \emph{reduced, irreducible} compact complex-analytic curve, and let $X \subset Y$ be a proper open subset. We say that a section or cohomology class on $X$ is \emph{overconvergent} if it arises by restriction from some strictly larger open $X' \subset Y$ satisfying $\ovl X \subset X'$.

\begin{proposition}\label{prop:OverconvergentImpliesTrivial}
    If $\mcF$ is a coherent sheaf on $Y$, then any overconvergent $s \in H^i(X, \mcF)$, with $i \geq 1$, is trivial. If $\ovl X \neq Y$, then any overconvergent $s \in H^1(X, \GL_n)$ is trivial.
\end{proposition}

\begin{proof}
    First, suppose that $\ovl X \neq Y$. By assumption $s$ arises from some $X'$ with $\ovl X \subset X'$; since $\ovl X \neq Y$, we may assume after shrinking that $X' \neq Y$ as well (remove a small closed neighborhood of some $y \notin \ovl X$). By \cite[Theorem~1]{NarasimhanSteinNormalization}, $X'$ is Stein, so we have $H^i_{\an}(X', \mcF) = 0$ for $\mcF$ a coherent sheaf on $X'$. Furthermore, $X'$ has the topological type of a 1-complex, so by the Oka principle \cite[Theorem~5.3.1]{ForstnericSteinManifolds}, we have
    \[ H^1_{\an}(X', \GL_n) =  \{1 \}. \]
    Letting $\mcF = \GL_n$ or a coherent sheaf on $Y$, suppose that $s$ is given by a \v{C}ech cocycle
    \[ (s_I) \in  \check Z^i((U_i), \mcF),  \]
    where $(U_i)$ is some cover of $X'$. We have just shown that $s_I$ is analytically a coboundary: There exists a non-definable cochain
    \[ t \in \check C^{i-1}_{\an}((U_i), \mcF) \]
    with $\partial t = s$. By compactness of $\ovl X$, there exists a refinement $(U_i')$ with $U_i' \Subset U_i$ whose union still covers $X$. The cochain $t$ is definable on $(U_i')$, since the structure contains $\bbR_{\an}$. Therefore, the restricted cocycle is a coboundary.

    Otherwise, suppose that $\ovl X = Y$ and let $\mcF$ be a coherent sheaf on $Y$. Choose some $p \in Y \setminus X$. Since restriction is functorial, it's enough to show that $H^i(Y, \mcF) \to H^i(Y \setminus \{p\}, \mcF)$ is the zero map for all $p \in Y$. By GAGA, $Y$ arises from an algebraic curve, and replacing $X$ by $Y \setminus \{p \}$, we can assume $X$ arises from an affine algebraic curve. Using \v{C}ech cocycles, define comparison maps
    $$ H^i_{\alg}(Y, \mcF) \xr{\alpha} H^i(Y, \mcF) \xr{\beta} H^i_{\an}(Y, \mcF). $$
    Given any analytic cover $(U_i)$ of $Y$, it has a finite definable refinement, and vice versa. If we also require the refinement $(U_i')$ to satisfy $\ovl{U_i'} \subset U_i$, then analytic cochains become definable cochains on the refinement since $S$ contains $\bbR_{\an}$. This shows, passing to the limit under refinement, that $\beta$ is an isomorphism. By GAGA, the composition $\beta \circ \alpha$ is an isomorphism, so it follows that $\alpha$ is an isomorphism as well. We have the following commutative diagram:
\[
\begin{tikzcd}[column sep=large, row sep=large]
    H^i_{\alg}(Y, \mcF) \arrow[r, "\cong"] \arrow[d, "\rho_{\alg}"'] &
    H^i(Y, \mcF) \arrow[r, "\cong"] \arrow[d, "\rho"] &
    H^i_{\an}(Y, \mcF) \arrow[d, "\rho_{\an}"] \\
    H^i_{\alg}(X, \mcF) \arrow[r, "\alpha_{X}"'] &
    H^i(X, \mcF) \arrow[r, "\beta_{X}"'] &
    H^i_{\an}(X, \mcF)
\end{tikzcd}
\]
Since $X$ is affine, we have $H^i_{\alg}(X, \mcF) = 0$. Chasing the diagram, we conclude that the map $\rho$ equals zero. \qedhere
\end{proof}

\begin{corollary}\label{cor:CollarVanishingSuffices}
    With the same conventions as \cref{prop:OverconvergentImpliesTrivial}, suppose that $s \in H^i(X, \mcF)$ is trivial on a ``boundary collar'' of the form $C \cap X$, where $C$ is a neighborhood of $\partial X$. Then $s$ is trivial.
\end{corollary}

\begin{proof}
    Pick a cover of the form
    \[ C \cap X, A_1, \ldots, A_k \]
    on which $s$ is represented by the cocycle $s_I$. Replacing $C \cap X$ by $C$ introduces no new overlaps, and gives an extension of $s$ to some $X' \supset \ovl X$. Therefore, we conclude by \cref{prop:OverconvergentImpliesTrivial}. \qedhere
\end{proof}

\begin{corollary}\label{cor:HigherStructureSheafVanishing}
    If $\mcF$ is any coherent sheaf on $Y$ (for example, $\mcO$), then we have $H^i(X, \mcF) = 0$ for $i \geq 2$.
\end{corollary}

\begin{proof}
    Let $s \in H^i(X, \mcF)$, and let it be represented by a cocycle $s_I$ on the cover $(U_i)$. By \cref{lem:NoTripleIntersect}, after refining $(U_i)$ we may assume that it has no triple intersections near $\partial X$. This implies that $s$ is trivial on a boundary collar, and we conclude by \cref{cor:CollarVanishingSuffices}. \qedhere
\end{proof}

\subsection{Reduction to the boundary}

Now, assume $S$ is unary-analytic and contains $\bbR_{\an}$. In this section, we show that for $H^1(X, \mcO)$, $H^1(X, \GL_n)$, triviality of cocycles is a ``local'' condition along $\partial X$. We require $Y$ to be smooth, but only on a neighborhood of $\partial X$, and as we now observe, this assumption is harmless.

\begin{proposition}\label{prop:ResolveExtraneousSingularities}
    There exists an open embedding $X \inj Y'$, where $Y'$ is a reduced irreducible compact complex-analytic curve smooth near $\partial X$.
\end{proposition}

\begin{proof}
    Let $\Sigma = \Sing(Y) \cap X^c$, and let
    \[ \pi \colon \tilde{Y} \to Y \]
    be a birational proper map resolving precisely the singularities in $\Sigma$. Since $S$ contains $\bbR_{\an}$, the space $\tilde{Y}$ and the map $\pi$ are definable, since they are locally definable over each $s \in \Sigma$. But $\pi$ is an isomorphism outside $\Sigma$, so there is an embedding $X \inj \tilde{Y}$ lifting the map $X \to Y$. \qedhere
\end{proof}

We will also assume that $X$ is \emph{boundary-connected}. This assumption is a priori stronger, but as we will see further on it can be removed in the proof of the main theorem.

\begin{definition}
    The class $s \in H^i(X, \mcF)$ is \emph{boundary-trivial} if there exists a boundary cover $(U_i)$ such that $s|_{U_i}$ is trivial.
\end{definition}

\begin{proposition}\label{prop:BoundaryTrivialImpliesTrivial}
    Assume that $X$ is boundary-connected. If $s \in H^1(X, \mcO)$ is boundary-trivial, then $s$ is trivial. The same holds for $H^1(X, \GL_n)$ as long as $Y \setminus X$ is infinite.
\end{proposition}

Observe that if $Y \setminus X$ is infinite and $X$ is boundary-connected, then we have $\ovl X \neq Y$. Otherwise, $Y \setminus X$ would contain a slit $C$ along which boundary-connectedness fails.

\begin{proof}
    First, consider $\mcO$: Let $(U_i)$ be a boundary cover on which $s$ is trivial, and consider the corresponding cocycle $(s_{ij})$. Applying \cref{lem:NoTripleIntersect}, we may assume that $(U_i)$ has no triple intersections. Now by \cref{prop:BoundaryAgreement}, if $x \in \partial U_{ij} \cap \partial X$, then with finitely many exceptions, we have $U_{ij} = X$ near $x$. Thus, since $X$ is boundary-connected, $s_{ij}$ is singular at only finitely many points of $\partial X$ by \cref{prop:GenericHolExtension}. Using \cref{lem:FiniteOmissionsLemma}, we can shrink $(U_i)$ and assume that each $s_{ij}$ is overconvergent. Applying \cref{lem:CoverExtension}, we obtain a cover $(V_i)$ of $\partial X$ with the following properties:
    \begin{enumerate}
        \item We have $V_i \cap X = U_i$.
        \item For all $i, j$, the function $s_{ij}$ extends definably to $V_i \cap V_j$.
    \end{enumerate}
    Applying \cref{lem:NoTripleIntersect}, assume that $(V_i)$ has no triple intersections. Then $(s_{ij})$ extends to a cocycle on $(V_i)$, and thus gives a class $t \in H^1(V, \mcO)$, where $V$ is some neighborhood of $\partial X$. By compactness of $\partial X$ and \cref{prop:OverconvergentImpliesTrivial}, we can replace $V$ with a smaller neighborhood and assume $t$ is trivial. But $C:= V \cap X$ is a boundary collar and on $C$ we have $s = t$. By \cref{cor:CollarVanishingSuffices}, $s$ is trivial on $X$.

    The argument for $\GL_n$ is much the same. Indeed the transition matrix $s_{ij}$ is overconvergent, as a matrix, at all but finitely many points. The determinant of its extension may vanish, but only at a proper definable analytic set, which is therefore finite. We can therefore apply \cref{lem:FiniteOmissionsLemma} again. Since $\ovl X \neq Y$, we conclude by invoking \cref{prop:OverconvergentImpliesTrivial}. \qedhere
\end{proof}

\begin{proposition}\label{prop:H1Vanishing}
    Suppose that $S$ is a branching structure containing $\bbR_{\an}$, and that $X$ is boundary-connected without punctures. Then both $H^1(X, \mcO)$ and $H^1(X, \GL_n)$ vanish.
\end{proposition}

\begin{proof}
    Let $s \in H^1(X, \mcF)$, where $\mcF = \mcO$ or $\GL_n$. Since $X$ is noncompact and has no punctures, $Y \setminus X$ is certainly infinite. Therefore, by \cref{prop:BoundaryTrivialImpliesTrivial}, we only need to produce, for every $x \in \partial X$, a boundary neighborhood $U_x$ on which $s$ is trivial. Working locally, we can identify $x$ with $0$ and $X$ with a boundary-connected subset of $\bbD^*$.

    Choose a cover $(U_i)$ on which $s$ is trivial. Following \cref{lem:ConnectednessAtZero}, we can refine so the $(U_i)$ are radial intervals. If $U_{ij}$ is nonempty near zero, then by \cref{prop:BranchingExtension}, the transition function $s_{ij}$ extends to a boundary neighborhood of $0$. Shrinking near $0$ we can assume $s_{ij}$ is a global section, and (in the $\GL_n$ case) that its determinant is nonvanishing. A cocycle whose transition function globalizes is manifestly trivial, so $s$ is trivial on $U_i \cup U_j$.

    Since the $U_i, U_j$ are radial intervals with nontrivial overlap near zero, their union $U_i \cup U_j$ is also a radial interval near zero. Thus, we can replace $U_i, U_j$ by $U_i \cup U_j$ in the cover and continue by induction. Eventually, we arrive at a cover such that $x \notin \partial U_{ij}$ for any $i, j$. But by local connectedness at $x$, it follows that $x$ only lies on the boundary of a single $U_i$, which means $U_i$ is an interior neighborhood of $x$. \qedhere
\end{proof}

\subsection{Proof of Theorem~A}\label{subsec:TheoremAProof}

We first dispense with the case where $X$ is nonreduced.

\begin{proposition}\label{prop:AssumeReduced}
    Let $X$ be any definable analytic space of dimension one, and let $X^{\red} \subset X$ be its reduction (cf. \cite[Section~2.7]{BakkerBrunebarbeTsimerman}). Suppose that all coherent sheaves on $X^{\red}$ are acyclic, and all vector bundles are trivial. Then the same holds for $X$ as well.
\end{proposition}

\begin{proof}
    Let $i \colon X^{\red} \to X$ be the inclusion, and let $\mcN_X$ be the nilradical, a coherent sheaf of ideals corresponding to the subspace $X^{\red}$. By \cite[Corollary~2.50]{BakkerBrunebarbeTsimerman}, $\mcN_X$ is nilpotent, so we have a finite filtration
    \[ 0 \subset \mcN^k \mcF \subset \cdots \subset \mcF, \]
    each of whose graded pieces is an $\mcO_{X^{\red}}$-module (note that the map $i$ is topologically a homeomorphism). Cohomology commutes with $i_*$ since $i$ is a closed immersion, so each graded piece is acyclic; it follows that $\mcF$ is acyclic.

    Now, let $\mcE$ be a vector bundle on $X$. By assumption, its pullback $i^* \mcE$ is trivial, so we have a short exact sequence
    \[ 0 \to \mcN_X \mcE \to \mcE \to \mcO_{X^{\red}}^n \to 0. \]
    Since $\mcN_X \mcE$ is acyclic, the sequence remains exact under $H^0$. Lifting the basis to $\mcE$, we obtain a map $\eta \colon \mcO_X^n \to \mcE$ which is an isomorphism after pulling back to $X^{\red}$. It follows, if we locally trivialize $\mcE$, that the determinant of $\eta$ is everywhere a unit mod $\mcN_X$. But this implies that the determinant is also a unit before quotienting by $\mcN_X$, since $\mcN_X$ is nilpotent. We conclude that $\eta$ is an isomorphism. \qedhere
\end{proof}

The next result will be used, in the case where $X$ is smooth, to reduce to the case of a torsion-free sheaf.

\begin{proposition}\label{prop:TorsionSheaf}
    Assume $X$ is reduced, and let $\mcF$ be any coherent sheaf on $X$. Then there is a coherent subsheaf $\mcT$, the \emph{torsion subsheaf}, such that $\mcF / \mcT$ is torsion-free. If $X$ is smooth, then $\mcF / \mcT$ is locally free.
\end{proposition}

\begin{proof}
    For any coherent sheaf, we can form the dual sheaf $\mcF^{\vee} = \shom(\mcF, \mcO)$. Since $\mcO$ is coherent, $\mcF^{\vee}$ is coherent as well. There is a canonical map $\mcF \to \mcF^{\vee \vee}$, and we define $\mcT$ to be its kernel; thus $\mcT$ is coherent. Next, recall \cite[Theorem~2.27]{BakkerBrunebarbeTsimerman} that there is an exact, faithful analytification functor
    \[ \Coh(X) \to \Coh(X^{\an}). \]
    Analytically, the above construction recovers the usual torsion sheaf $\mcT^{\an} \subset \mcF^{\an}$. Therefore, the quotient $(\mcF / \mcT)^{\an}$ is torsion-free, which implies that $\mcF / \mcT$ is as well.

    Finally, suppose that $X$ is smooth, and let $\mcE = \mcF / \mcT$. Then $\mcE^{\an}$ is locally free of rank $k$, since the local rings of $X^{\an}$ are discrete valuation rings. Passing to a cover, we may assume that $\mcE$ is locally generated. For any tuple $s_1, \ldots, s_k$ of length $k$, the set $U$ where they generate $\mcE$ is open, and $\mcE$ over $U$ is free of rank $k$. By comparing with analytic stalks we see that these $U$ form a cover, and since the cover is finite and definable, we conclude. \qedhere
\end{proof}

In the singular case we will also need the following basic observation, applied to the normalization $\nu \colon \tilde{X} \to X$. Since $S$ contains $\bbR_{\an}$, the normalization is definable.

\begin{lemma}\label{lem:FinitePushforward}
    Let $f \colon X \to Y$ be any proper map with finite fibers, where $X, Y$ are definable topological spaces. Then sheaf cohomology commutes with the pushforward $f_*$.
\end{lemma}

\begin{proof}
    By \cite[Corollary~2.9]{BakkerBrunebarbeTsimerman}, $f_*$ is exact on the category of sheaves. One can compute cohomology using flasque resolutions and the pushforward of a flasque sheaf is flasque, so the result follows formally. \qedhere
\end{proof}

\begin{proof}[Proof of \cref{thm:Vanishing}]
    By \cref{prop:AssumeReduced}, we can assume that $X$ is reduced, and by \cref{prop:ResolveExtraneousSingularities}, we can assume $\Sing(Y) \subset X$. By \cref{cor:TameEmbeddability}, there is an embedding
    \[ \tilde{X} \to Y' \]
    which makes $\tilde{X}$ boundary-connected. Choose a boundary collar $C \subset X$ such that $\tilde{C} \to C$ is an isomorphism: Then there exists some $D \subset Y'$ with $D \cup \tilde{X} = Y'$ and $D \cap \tilde{X} = \tilde{C}$. Gluing $D$ to $X$ along $C = \tilde{C}$, we can assume that $X$ is boundary-connected. It follows, by \cref{prop:H1Vanishing}, that $H^1(X, \mcO)$ and $H^1(X, \GL_n)$ vanish.

    It remains to show that every coherent sheaf $\mcF$ is acyclic. By \cref{prop:TorsionSheaf}, we have a short exact sequence
    \[ 0 \to \mcT \to \mcF \to \mcV \to 0, \]
    where $\mcT$ is torsion and $\mcV$ is torsion-free. The support of $\mcT$ is zero-dimensional, hence finite, and thus $\mcT$ is flasque. Therefore, we reduce to the case when $\mcF$ itself is torsion-free. First suppose $X$ is smooth: Then $\mcF$ is a vector bundle, and all vector bundles are trivial so we have $\mcF \cong \mcO^n$. By \cref{prop:H1Vanishing} and \cref{cor:HigherStructureSheafVanishing}, $\mcF$ is acyclic.

    To handle the singular case, we loosely adapt an argument of Narasimhan \cite[Section~3]{NarasimhanSteinNormalization}. Let $\nu \colon \tilde{X} \to X$ be the normalization, and let $\mcE = \nu^* \mcF$. Since $\mcF$ is torsion-free, we have an exact sequence
    \[ 0 \to \mcF \to \nu_* \mcE \to \mcG \to 0, \]
    and since $\Sing(X)$ is finite, the cokernel $\mcG$ has finite support. Thus $\mcG$ is acyclic. Since $\tilde{X} \subset \tilde{Y}$ has no punctures, $\mcE$ is acyclic, and by \cref{lem:FinitePushforward}, so is $\nu_* \mcE$. By the long exact sequence in cohomology, it suffices to show that the map
    \[ \eta \colon H^0(\tilde{X}, \mcE) = H^0(X, \nu_* \mcE) \to H^0(X, \mcG) \]
    is surjective. Let $\mcI$ be the annihilator of $\mcG$, a coherent $\mcO_X$-ideal, and let $\mcJ = \mcI \cdot \mcO_{\tilde{X}}$. Then $\mcJ$ is coherent, because it is locally a finitely-generated submodule of $\mcO_{\tilde{X}}$. Furthermore, the induced map $\nu_* (\mcJ \cdot \mcE) \to \mcG$ is the zero map. It follows that $\eta$ has a factorization
    \[ H^0(X, \nu_*\mcE) \to H^0(X, \nu_*(\mcE / \mcJ \mcE)) \to H^0(X, \mcG). \]
    By acyclicity on $\tilde{X}$, the first map is a surjection. But the latter two sheaves have finite support, so the second map surjects as well, and we conclude. \qedhere
\end{proof}

\section{Comments on the vanishing theorem}\label{sec:Comments}

\subsection{Consequences of cohomology vanishing}

In the analytic setting, cohomology vanishing (Cartan's theorem~B) has a variety of consequences, and we now survey to what extent the same implications hold definably. We first show, following a classical observation of Siu \cite{SiuCartan}, that the analogue of Cartan's theorem~A holds.

\begin{proof}[Proof of \cref{cor:CartanA}]
    By \cite[Proposition~2.46]{BakkerBrunebarbeTsimerman}, coherent subsheaves of $\mcF$ satisfy the ascending chain condition. Thus, it is enough to argue by induction and show the following: Given any map $ \mcO^m \xr{g} \mcF$, either $g$ is surjective or there is a global section of $\mcF$ not in the image of $g$. By \cite[Corollary~2.40]{BakkerBrunebarbeTsimerman}, exactness of coherent sheaves can be checked using stalks at \emph{ordinary} points. Therefore, if $g$ is not surjective, there is some $x \in X$ for which $g_x$ is not surjective. On the other hand, we have a surjection of sheaves $\mcF \to \mcF \otimes \bbC_x$, and therefore, by cohomology vanishing applied to the kernel, a surjection
    \[ H^0(X, \mcF) \to H^0(X, \mcF \otimes \bbC_x) \]
    of global sections. We can therefore choose a global section not in the image of $H^0(g)$ and conclude. \qedhere
\end{proof}

\begin{remark}
    Suppose that $X^{\df}$ is the definabilization of an algebraic space $X$, and that coherent cohomology vanishes on $X^{\df}$. Using \cref{cor:CartanA}, one could then show that the functor
    $$ \Coh(X) \to \Coh(X^{\df}) $$
    is an equivalence of categories, obtaining a stronger version of the o-minimal GAGA theorem. Indeed, if $\mcF$ is a coherent sheaf we obtain a global presentation
    $$ \mcO^m \xr{f} \mcO^n \to \mcF \to 0, $$
    and the map $f$ is algebraic by o-minimal GAGA, so the cokernel $\mcF$ is algebraic as well. However, it appears that algebraic spaces \emph{never} satisfy definable coherent vanishing, at least in dimension one.
\end{remark}

\begin{remark}

For an ordinary open set $U \subset \bbC^n$, the following are equivalent characterizations of a \emph{Stein domain:}

\begin{enumerate}
    \item All coherent sheaves on $U$ are acyclic.
    \item There exists a holomorphic $f \colon U \to \bbC$ that extends to no larger domain.
    \item There exists a holomorphic closed embedding $U \to \bbC^N$, for some $N$.
\end{enumerate}

However, the corresponding equivalence for definable analytic spaces is false. Indeed, if $S$ is branching and contains $\bbR_{\an}$, then the unit disk $\bbD$ satisfies condition (1) by \cref{thm:Vanishing}. On the other hand, for a boundary-connected open subset $U \subset \bbC$, conditions (2) and (3) will usually imply that $\bbC \setminus U$ is finite. Suppose $U$ is boundary-connected, $\bbC \setminus U$ is infinite (e.g., $U$ is the unit disk), and $S$ is unary-analytic. Then (2) cannot hold because it contradicts \cref{prop:GenericHolExtension}. For condition (3), one can work in any structure and use the definable Chow theorem to see that $U$ is algebraic.
\end{remark}

\subsection{Sharpness of hypotheses}

In this section, we show by example that all the hypotheses of \cref{thm:Vanishing} are necessary. By the proof of \cref{cor:CartanA}, if all coherent sheaves on a space $X$ are acyclic, then every line bundle is generated by global sections. We will exhibit a series of line bundles with no nonzero global sections, thus falsifying both halves of the theorem.

\begin{lemma}\label{lem:expnotdefinable}
    Let $X$ be a connected definable set, and let $f \colon X \to \bbC$ be a continuous function (not necessarily definable) whose image has unbounded imaginary part. Then $e^{f(z)}$ is not definable.
\end{lemma}

\begin{proof}
    If $e^{f(z)}$ is definable, then so is $g(z) := e^{f(z)} / |e^{f(z)}|$. We have
    \[ g^{-1}(1) = \bigcup_n (\Im f)^{-1}(2 \pi n). \]
    Since $X$ is connected and $f$ has unbounded imaginary part, infinitely many of these sets are nonempty. It follows that $g^{-1}(1)$ has infinitely many connected components, contradicting definability. \qedhere
\end{proof}

Our first counterexample is a slight modification of \cite[Example~3.2]{BakkerBrunebarbeTsimerman}. It shows that one can never remove the puncture hypothesis: The punctured disk has nontrivial line bundles in \emph{any} o-minimal structure.

\begin{proposition}\label{prop:G_mLineBundle}
    Let $S$ be any o-minimal structure, and let $X = D(0, r) \setminus \{0 \}$, where $0 < r \leq \infty$; if $r = \infty$ we have $X = \bbC \setminus  \{0 \}$. Then there exists a definable line bundle $\mcL$ on $X$ with no nonzero global sections.
\end{proposition}

\begin{proof}
    Cover $X$ by three radial sectors $U_1, U_2, U_3$ with only twofold intersections, and define $\mcL$ by $s_{12} = s_{23} = 1, s_{31} = 2$. A global section is a holomorphic function on $X$ with monodromy $2$ around zero. Analytically, we have the global section $e^{\alpha \log z}$, with $\alpha = \frac{\log 2}{2 \pi i}$.
    Suppose that $f \in H^0(X, \mcL)$ is nonzero. Its zero set is definable, hence finite, so after shrinking $r$ we can assume $f$ is nonvanishing. Since $f$ has monodromy $w \mapsto 2w$ around the origin, $\log f$ has monodromy $w \mapsto w + \log(2) + 2 \pi i n$ for some $n$. Replacing $f$ by $z^{-n} f$, we can assume its monodromy is purely real.

    It follows that $u := \Im(\log(f))$ is a single-valued, bounded, harmonic function on $\bbD^*(0,r)$, and therefore extends across the puncture. But then so does its harmonic conjugate $-\Re(\log(f))$. Since the latter function has nontrivial monodromy, we reach a contradiction. \qedhere
\end{proof}

The next counterexample, which is essentially folklore, shows that the branching hypothesis on $S$ is needed.

\begin{proposition}
    Let $S$ be $\bbR_{\an,\exp}$, or any structure that defines both the real exponential and the restricted sine function. For $t \in [-\infty, \infty)$, let
    \[ X = \{z: \Re(z) > t \}. \]
    Then there exists a definable line bundle $\mcL$ on $X$ with no nonzero global sections.
\end{proposition}

\begin{remark}
    In particular, since the Cayley transform is definable, the proposition applies to the open unit disk $\bbD$.
\end{remark}

\begin{proof}
    Take a cover by the sets
    \[ U_+ = \{z \in X : \Im(z) > -1 \}, U_- = \{z \in X: \Im(z) < 1 \}. \]
    Take the transition function $z \mapsto \exp(z)$. Since the overlap has bounded imaginary part, $e^z$ is definable in this structure and we obtain a line bundle $\mcL$.

    Suppose now that $s \in H^0(X, \mcL)$ is a nonzero global section; this corresponds to a tuple $s_+ \in \mcO(U_+)$, $s_- \in \mcO(U_-)$ such that $s_+ / s_- = \exp$. Both $s_+$ and $s_-$ have finite zero sets by o-minimality, so at the cost of modifying $t$ we may assume they are nonvanishing.

    Since $\exp$ is holomorphic on $X$, it follows that $s_+, s_-$ extend holomorphically to $X$ as well, though not definably. Since all sets are simply connected, we can pick single-valued logarithms, and we have
    \[  \log(s_+) - \log(s_-) = z. \]
    By \cref{lem:expnotdefinable}, $\log(s_+)$ has bounded imaginary part on $U_+$, and $\log(s_-)$ has bounded imaginary part on $U_-$. Thus, we have a harmonic function $h(z)=  \Im(\log(s_-(z)))$ which is bounded on $U_-$ and grows linearly on $U_+$. For large $r \gg 0$, we have
    \[ h(r) =  \frac{1}{\mu(D(r, r /2))} \int_{D(r, r / 2)} h(z) dV, \]
    by the mean value property, and the right-hand side grows linearly in $r$. But this contradicts the assumption that $h$ is bounded on $U_-$. \qedhere
\end{proof}

\begin{remark}
    Similar counterexamples are also developed at greater length in a forthcoming paper of Brosnan and Melrod.
\end{remark}

Finally, we show that the structure $S$ must contain $\bbR_{\an}$. The structure $\bbR_{\alg}$ \emph{is} branching, but in itself this is not enough for cohomology vanishing.

\begin{proposition}
    Let $S = \bbR_{\alg}$. Then there exists a line bundle on the annulus
    \[ X:= \{z: 1 < |z| < 2 \} \]
    with no nonzero global sections.
\end{proposition}

\begin{proof}
    Define, as in \cref{prop:G_mLineBundle}, a line bundle $\mcL$ whose sections are multivalued functions, with monodromy $f(z) \mapsto 2 f(z)$ around the origin. An algebraic function has finite-order monodromy around any loop, so $\mcL$ has no nonzero global sections. \qedhere
\end{proof}

\section{Cohomology of the structure sheaf}\label{sec:StructureSheaf}

In this section, we prove \cref{thm:computation}. Our standing assumption is that $S$ is a branching structure containing $\bbR_{\an}$, and that $X$ is a noncompact open subset of the compact Riemann surface $Y$. A \emph{puncture} of $X$ is an isolated point of its complement.

\begin{proposition}\label{prop:CohomologyPunctureDirectSum}
    We have an isomorphism
    \[ H^1(X, \mcO) \cong \bigoplus_{i=1}^n H^1(\bbD^*, \mcO),  \]
    where $n$ is the number of punctures.
\end{proposition}

\begin{proof}
    Let $X'\supset X$ be the open subset of $Y$ obtained by filling the punctures. Letting $V$ be a union of disjoint disks around each puncture, we have a cover $(X, V)$ of $X'$, and we can apply Mayer--Vietoris (\cref{prop:MayerVietoris}) to obtain the following exact sequence:
    \[ H^1(X', \mcO) \to H^1(X, \mcO) \times \bigoplus_i H^1(\bbD, \mcO) \to \bigoplus_i H^1(\bbD^*, \mcO) \to H^2(X', \mcO). \]
    We have $H^1(\bbD, \mcO) = 0$ by \cref{thm:Vanishing}. If $X' \neq Y$, then the first and last terms vanish by \cref{thm:Vanishing}, since $X'$ is noncompact and has no punctures. If $X' = Y$, then the final term still vanishes, and by \cref{prop:OverconvergentImpliesTrivial}, the initial map is the zero map. Either way, we obtain the desired isomorphism. \qedhere
\end{proof}

Let us define
\[ \mcO_{0^+} = \varinjlim_{r > 0} \Def((0,r), \bbC). \]
An element of $\mcO_{0^+}$ is precisely the germ of a definable function near $0^+$. Since $S$ is branching, every $f \in \mcO_{0^+}$ gives rise to a multivalued function with definable branches. The action of a counterclockwise loop gives a well-defined \emph{monodromy operator}
\[ M \colon \mcO_{0^+} \to \mcO_{0^+}. \]
We can likewise define
\[ \mcO_{\bbD^*, 0} = \varinjlim_{r > 0} \mcO(\bbD^*(0,r)). \]

\begin{example}
    Suppose that $S = \bbR_{\an}$. Then
    \[ \mcO_{\bbD^*, 0} \cong \bbC \{z \} [z^{-1}] \]
    is the space of convergent Laurent series, and
    \[ \mcO_{0^+} \cong \bigcup_n \bbC \{z^{1 / n} \} [z^{-1 / n}] \]
    is the space of convergent Puiseux series.
\end{example}

\begin{proposition}\label{prop:MonodromyExactSequence}
    We have an exact sequence
    \[ 0 \to \mcO_{\bbD^*, 0} \xr{i} \mcO_{0^+} \xr{\gamma} \mcO_{0^+} \xr{\pi} H^1(\bbD^*, \mcO) \to 0, \]
    where $\gamma = \id-M$ is the monodromy difference operator.
\end{proposition}

\begin{proof}
    Every map is evident except for the map $\pi$, which we now describe. By applying Mayer--Vietoris in the same fashion as above, we see that the restriction map $H^1(D^*(0,r), \mcO) \to H^1(D^*(0,s), \mcO)$ is an isomorphism for all $r > s$. Thus, we can interpret the final term equivalently as the direct limit
    \[ \varinjlim_{r \to 0} H^1(D^*(0,r), \mcO). \]
    For any $r$, consider the following two-element cover of the punctured disk $D(0,r)^*$:
    \[
    \begin{gathered}
        U_r = D(0,r) \setminus [0,r), \\
        V_r = \{z \in D(0,r)^*: \Arg(z) \in (-\varepsilon, \varepsilon) \}.
    \end{gathered}
    \]
    \begin{figure}[ht]
    \centering
    \includegraphics[width=0.5\linewidth,
    alt={A disk U r with center 0 marked by a small open circle and a slit along the positive real radius from 0 to r. A thin orange sector V r opens around the slit, from 0 out to the rim. The two thin regions where V r overlaps U r, one above the slit and one below, are shaded and labeled W r plus and W r minus.}]{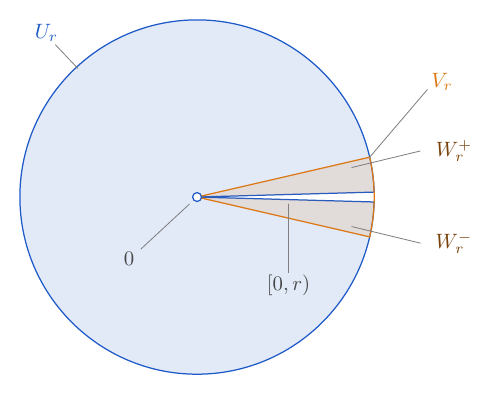}
    \caption{The cover used to prove \cref{prop:MonodromyExactSequence}.}
    \label{fig:slit-disk-cover}
    \end{figure}

    The overlap $U_r \cap V_r$ has an upper and lower component, which we denote $W^+_r, W^-_r$. Fix a germ $g \in \mcO_{0^+}$, which must extend to $V_r$ for some $r$. Then we let $\pi$ be the cocycle, for the above cover, whose transition is $0$ on $W^+_r$, and $g$ on $W^-_r$.

    If $s \in H^1(D^*(0,r), \mcO)$, then $s$ is trivial on both $U_r$ and $V_r$ by \cref{thm:Vanishing}, so it is represented by transition functions $s_-, s_+$. Since $S$ is branching, after shrinking $r$ we can assume that $s_+$ extends over $V_r$, so by subtracting a coboundary we can assume that $s_+ = 0$. This shows that $\pi$ is surjective.

    Next, we show that $\ker(\pi) = \im(\gamma)$. In one direction, suppose that $g = f - M(f)$ is a monodromy difference. Consider the cochain whose value on $V_r$ is $f$, and whose value on $U_r$ is the counterclockwise continuation: One readily observes that its coboundary is $\pi(g)$. Conversely, suppose that $g$ is the coboundary of some $a \in \mcO(V_r), b \in \mcO(U_r)$. Then $a, b$ must agree on $W^+_r$, so we can view $b$ as an analytic continuation, and the expression $a - b$ on $W^-_r$ is a monodromy difference.

    It remains to show exactness at the first two terms. But $i$ is injective by the identity theorem, and the identity $\ker(\gamma) = \im(i)$ simply says the following: $f$ is holomorphic on some punctured disk if and only if its continuation is single-valued. \qedhere
\end{proof}

Now, specialize to the structure $S = \bbR_{\an}^K$. For a tuple $(s_1, \ldots, s_m) \in \bbC^m$, consider the following Diophantine approximation condition:

\begin{enumerate}
    \item[$(\star)$] For all integer polynomials $P \in \bbZ[x_1, \ldots, x_m]$ such that $P(s) \neq 0$, we have
    \[ |P(s)| \geq e^{- A H}, \]
    where $H = H(P)$ is the height of $P$ and $A > 0$ is a fixed constant that depends on $d = \deg(P)$ and $s$.
\end{enumerate}

\begin{proposition}
    There exists an injective comparison map
    \[ \eta \colon \bbC \{z \} [z^{-1}] \to H^1(\bbD^*, \mcO). \]
    If every finite subset $A \subset K$ is contained in a field $\bbQ(s_1, \ldots, s_m)$, where $(s_i)$ satisfies $(\star)$, then $\eta$ is surjective.
\end{proposition}

\begin{proof}
    Since $S$ contains $\bbR_{\an}$, every Laurent series is definable near zero, and the map $\eta$ is simply the restriction of the map $\pi$ constructed above.

    By \cite[Proposition~4.5]{MillerPowerFunctions}, every $\bbR_{\an}^K$-germ admits a convergent generalized power series expression
    \[f(z) = \sum_I a_I z^{r \cdot I}, r \cdot I = r_0 + r_1 i_1 + \cdots + r_k i_k,  \]
    where $r_1, \ldots, r_k > 0$. We may assume without loss of generality that no two $r \cdot I$ have the same degree, by combining coefficients if necessary. The action of monodromy on the power function $z^r$ is straightforward, and we have
    \[ \gamma(f)(z) = \sum_I a_I (1-e^{2 \pi i (r \cdot I)}) z^{r \cdot I}. \]
    In particular, for any $f$, the integer-power coefficients of $\gamma(f)$ vanish. This shows that $\eta$ is injective. Furthermore, we can write
    \[ f = f_1 + f_2, \]
    where $f_1$ is a Laurent series, and all integer powers in $f_2$ vanish. Thus, $\eta$ is surjective if and only if, whenever $f$ lacks integer powers, it equals $\gamma(g)$ for some $g$. If this holds, then we must have
    \[ g(z) = \sum_I \frac{a_I}{1-e^{2 \pi i (r \cdot I)}} z^{r \cdot I}. \]
    Since $g$ exists formally, it represents a definable function if and only if it has positive radius of convergence. But since $e^{2\pi iz}$ has controlled derivative near the origin, up to a constant the coefficients have magnitude
    \[ |a_I| d(r \cdot I, \bbZ)^{-1}. \]
    Our goal is to show that this expression grows at most exponentially in $|I|$. This is true by assumption for the $a_I$, so we reduce to showing for sufficiently large $I$ that
    \[ |r \cdot I - k| \geq e^{- A |I|}, \]
    where $k$ is the nearest integer.

    Now, suppose that the $r_i$ are contained in a finitely generated field $\bbQ(s_j)$, where $(s_i)$ satisfies $(\star)$. We can write $r_i = f_i(s)$, where $f_i$ is a rational function. Fix a common denominator $q \in \bbQ[x_j]$, such that $f_i = p_i / q$, and such that $p_i, q$ are \emph{integer} polynomials. Then the desired bound for large $I$ becomes
    \[ \left| p_0(s) + \sum_j  p_j(s) \cdot i_j - q(s) \cdot k \right| \geq e^{- A |I|} |q(s)|. \]
    The expression $|q(s)|$ is a fixed nonzero constant independent of $I$, so we can disregard it. Thus, taking
    \[ P = p_0+ \sum_j i_j \cdot p_j - k \cdot q, \]
    the bound follows from $(\star)$. \qedhere
\end{proof}

\begin{proposition}
    Condition $(\star)$ holds in the following cases:
    \begin{enumerate}
        \item All the $s_i$ are algebraic numbers.
        \item All the $s_i$ are exponentials of algebraic numbers.
        \item $m = 1$ and $s_1 = \pi$.
    \end{enumerate}
\end{proposition}

\begin{proof}
    First, suppose the $s_i$ are algebraic. Then by Liouville's inequality \cite[Proposition~3.14]{WaldschmidtDiophantineApprox}, for $P(s) \neq 0$ we obtain a bound of the form
    \[ |P(s)| \geq H^{1-D} \cdot C, \]
    where $D = [ \bbQ(s_1, \ldots, s_m): \bbQ]$, and $C$ depends only on $s$ and the degree of $P$.

    Next, suppose $s_i = e^{q_i}$, where the $q_i$ are algebraic, and let $P = \sum_I a_I x^I$ be an integer polynomial. Then we can write
    \[ P(s) = \sum_I a_I e^{q \cdot I}. \]
    If the $q_i$ are $\bbQ$-linearly independent, then the Lindemann--Weierstrass theorem implies that $P(s) \neq 0$, and a stronger version due to Mahler \cite[Satz~2, pp.~132--133]{MahlerApproximationI} says that
    \[ |P(s)| \geq H^{-c d^n}, \]
    for an explicit constant $c$ depending on $q$. This is much stronger than the required exponential bound. If the $q_i$ are not $\bbQ$-linearly independent, clear denominators and choose a $\bbQ$-linearly independent tuple
    $(q'_1,\ldots,q'_r)$ of algebraic numbers such that each $q_i$ lies in the $\bbZ$-span of the
    $q'_\ell$. Writing $s'_\ell=e^{q'_\ell}$, the value $P(s)$ can be written as $R(s')$, where
    $R$ is an integer Laurent polynomial. For fixed $\deg P=d$, multiplying $R$ by a monomial
    depending only on $d$ and the relations among the $q_i$ gives an ordinary integer polynomial $P'$ with
    \[ P'(s') = (s')^B P(s). \]
    Moreover $\deg P'$ is bounded in terms of $d$ and the $q_i$, and $H(P') \leq C_d H(P)$.
    Thus, if $P(s)\neq 0$, applying the independent case to $P'$ gives the required lower bound
    for $P(s)$ after absorbing the fixed nonzero factor $(s')^B$.

    Finally, let $s = (\pi)$. Then by a further theorem of Mahler \cite[Satz~5]{MahlerApproximationII}, we have
    \[ |P(s)| \geq H^{-c^d}, \]
    where $c$ is an absolute constant. \qedhere
\end{proof}

\begin{example}
    Let $t > 0$ have the form
    \[ t = \sum_{n \geq 1} 10^{- a_n}, \]
    where the sequence $a_n$ grows very fast, for example $a_{n+1} = 10^{2 a_n}$. Then for $K = \bbQ(t)$, the map $\eta$ is not surjective: If we take the series
    \[ f(z) := \sum_{n \geq 1} z^{n t}, \]
    then the coefficients of the formal series $\gamma^{-1}(f)$ grow superexponentially. Indeed, for $N_m = 10^{a_m}$, the distance from $N_m t$ to an integer goes like $10^{a_m - a_{m+1}}$.
\end{example}

\begin{proposition}
    The set of $(t_1, \ldots, t_m) \in \bbR^m$ satisfying $(\star)$ has full measure in $\bbR^m$.
\end{proposition}

\begin{proof}
    If $(\star)$ holds for the tuple $t$, then it holds for any integer multiple $n t$ with $n > 0$. Therefore, it is enough to show full measure after restricting to $[-1,1]^m$. Fix for the moment an integer polynomial $P$ of degree $d$, and some $\delta > 0$. If we let $E = P^{-1}((-\delta, \delta)) \subset [-1,1]^m$, then an estimate due to Brudnyi and Ganzburg \cite[eq.~(14), p.~354]{BrudnyiGanzburg} gives
    \[ \sup |P| \leq C \mu(E)^{-d} \cdot \delta. \]
    In other words, as $\sup |P| \to \infty$, the set of points where $P$ is small shrinks in measure. By norm-equivalence on the space of degree $d$ polynomials, the left-hand side is comparable for fixed $d$ to $H(P)$. Rearranging, we obtain
    \[ \mu(E) \leq C \left( \frac{\delta}{H(P)} \right)^{1/d}. \]
    The number of polynomials with degree $d$ and height at most $N$ is $O(N^{\binom{m+d}{d}})$. Thus, the set of $s \in [-1,1]^m$ with $|P(s)| < \delta$ for some such $P$ has measure at most
    \[ C \delta^{1 / d} N^{\binom{m+d}{d} - 1/d}.  \]
    Fix some $k \gg 0$, and let $X_N$ be the set of $s \in [-1,1]^m$, such that for some $P$ of degree $d$ and height at most $N$, we have $|P(s)| < N^{-k}$. Then since $k$ is large, the series
    \[ \sum_N \mu(X_N) \]
    converges. By Borel--Cantelli, the set $\limsup_N X_N$ has measure zero. But if $s \notin \limsup_N X_N$, then it certainly satisfies $(\star)$. \qedhere
\end{proof}

This completes the proof of \cref{thm:computation}.

\section{Definable compactifications}\label{sec:Compactification}

Let $X$ be a Riemann surface: Our goal is to understand when $X$ definably embeds into a compact Riemann surface. We first give two examples to illustrate the possible obstructions. Then we define the central object of this section, the \emph{intrinsic closure} $\hat{X}$ of $X$. After verifying its basic properties, we use the intrinsic closure to prove \cref{thm:compactification}.

\subsection{Two obstructions}

The first example shows that monodromy around punctures obstructs definable compactification.

\begin{example}
    Let $S$ be any o-minimal structure. Take a radial cover $U_1, U_2, U_3$ of $\bbG_m$. Mimicking \cref{prop:G_mLineBundle}, we define a Riemann surface $X$ with charts $(U_i)$ and transition functions $s_{12} = s_{23} = z, s_{31} = 2 z$. We have a (non-definable) biholomorphism
    \[ \bbG_m \xr{\cong} X, z \mapsto z \cdot e^{\alpha \log z}, \]
    as in the proposition, which shows that the analytic end at $0$ is a punctured disk. Suppose for contradiction that $X$ can be definably compactified via some map $X \inj Y$. The puncture corresponds to a point $y \in Y$, and $y$ has a neighborhood isomorphic to the disk. We thus obtain a definable embedding
    \[ f \colon \bbD^* \to X. \]
    If we let $V_i = f^{-1}(U_i)$, then the map $f$ gives a multivalued function, single-valued on each $V_i$, with monodromy 2 around the origin. But as we showed in \cref{prop:G_mLineBundle}, no such function is definable.
\end{example}

The next example shows that even if $X$ \emph{analytically} has no punctures, it still may not be compactifiable. The problem is that there may be points which definably ``appear'' as punctures, even though analytically they represent disk ends.

\begin{figure}[ht]
  \centering
  \includegraphics[width=\linewidth,
    alt={Left: the strip U, the set of z with imaginary part between 1 and 2 plus epsilon, drawn above the real axis with its two ends tapering to points labeled infinity, one at each end. A curved arrow labeled z equivalent to 2 z points to the right-hand picture: the analytic surface X an, the upper half-plane modulo z equivalent to 2 z, drawn as an annulus. Colors match the edges and ends of the strip to the circles and seam of the annulus: the ends at infinity correspond to entire boundary circles, not to a puncture.}]{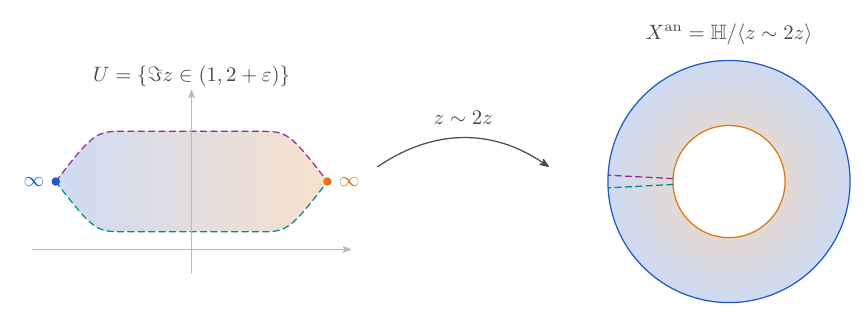}
  \caption{The idea of \cref{SecondCompactificationCounterexample}.}
  \label{fig:annulus-obstruction}
\end{figure}

\begin{example}\label{SecondCompactificationCounterexample}
    Again, let $S$ be arbitrary. Consider the strip
    \[ U = \{z \in \bbC: \Im(z) \in (1, 2+\varepsilon)\}, \]
    and glue it to itself via $z \mapsto 2z$ to obtain a definable Riemann surface $X$. If $X \inj Y$ were a definable compactification, then by \cref{prop:ctsext}, the map $U \to Y$ has finitely many limit points over each $z \in \partial U$. Aside from $\infty$, every boundary point of $U$ lies in the interior of $X$, so this would imply that $\ovl X \setminus X$ is finite. But analytically, we can view $X$ as
    \[ \bbH / \{z \sim 2 z \}, \]
    and by using the fundamental domain $\{z \in \bbH: |z| \in (1, 2+\varepsilon) \}$, we see that $X^{\an}$ is an annulus and has no punctures.
\end{example}

\subsection{Intrinsic closures}

In this section, we build a space $\hat{X}$ which captures these obstructions. Given an atlas $(U_i)$ for $X$, the idea is simply to glue the closures $\ovl{U_i}$ along continuous extensions of transition functions. But the space built in this way depends on the cover and can be poorly behaved. We identify a suitable class of covers cofinal under refinement, and then construct a well-behaved space $\hat{X}$ that is independent of choices.

\begin{definition}
    The point $x \in \partial U_i$ is a \emph{$j$-interior point} if $U_{ij}$ is a boundary neighborhood of $x$ in $U_i$.
\end{definition}

\begin{figure}[ht]
  \centering
  \includegraphics[width=\linewidth,
    alt={Two disk-shaped charts, U i on the left and U j on the right. A point x lies on the boundary circle of U i, inside the region where U i overlaps U j, with a dashed green neighborhood around it. A dashed arrow labeled s ij bar carries x to a point y in the interior of U j, which also has a dashed green neighborhood.}]{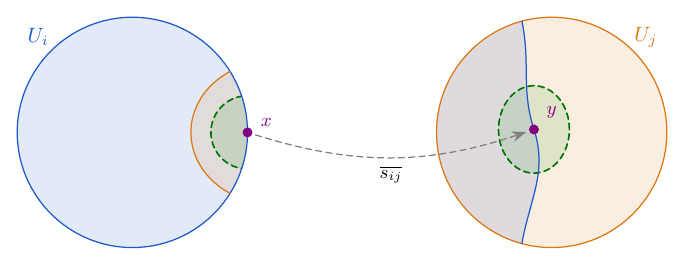}
  \caption{$j$-interior points.}
  \label{fig:j-interior}
\end{figure}

\begin{lemma}\label{lem:BoundaryInteriorIdentifications}\label{lem:ClosedExcision}
    Let $(U_i)$ be a cover of a Riemann surface $X$ by boundary-connected charts.
    \begin{enumerate}
        \item If $(x, y) \in \partial{\Gamma(s_{ij})}$ and $y \in U_j$, then $x$ is $j$-interior.
        \item Suppose $A \subset \ovl{U_i}$ is closed and contains no $j$-interior points. Then after shrinking $(U_j)$, we may assume that $A \cap U_{ij} = \emptyset$.
    \end{enumerate}
\end{lemma}

\begin{proof}
    Fix a triangulation compatible with the $U_i$ and with the transition maps, in the sense that $s_{ij}$ identifies simplices with simplices. We can assume after refining that $s_{ij}$ extends continuously to the closure of any simplex, and that the extensions to closures also identify simplices with simplices.

    \textbf{(1)}: Suppose by contradiction that $U_{ij}$ is not a boundary neighborhood. Since $U_i$ is boundary-connected, there exists a 1-simplex $\sigma \subset \partial U_{ij} \cap U_i$ meeting $x$. Choose a 2-simplex $\tau \subset U_{ij}$ which has $\sigma$ as a face: Then $\ovl{\Gamma(s_{ij})}$ identifies $\sigma$ with a face of $s_{ij}(\tau)$, therefore with a 1-simplex of $U_j$. Since $X$ is Hausdorff, we have
    \[ \ovl{\Gamma(s_{ij})} \cap (U_i \times U_j) = \Gamma(s_{ij}), \]
    and therefore $\sigma \subset U_{ij}$, which would be a contradiction.

    \textbf{(2)}: Let $A'$ be the set of $y \in U_j$ which are identified via $\ovl{\Gamma(s_{ij})}$ with $A$. By (1), we have $A' \subset U_{ji}$, and $A'$ is closed. Thus, we can simply replace $U_j$ by $U_j \setminus A'$, and the result still covers $X$. \qedhere
\end{proof}

\begin{definition}
    A \emph{boundary-tame} cover of $X$ is a cover $(U_i)$ by boundary-connected charts $U_i \subset \bbC$, whose double intersections $U_{ij}$ are also boundary-connected.
\end{definition}

Note that boundary-connectedness of $U_{ij}$ is well-defined, since we interpret it as a subset of $\bbC \subset \bbP^1$.

\begin{lemma}\label{lem:BoundaryTameRefinements}
    Every cover of $X$ has a boundary-tame refinement.
\end{lemma}

\begin{proof}
    Fix a cover $(U_i)$ of $X$, and refine so that the $(U_i)$ are charts. Fix triangulations of $\ovl{U_i}$ which are compatible with $s_{ij}$ and $\Gamma(s_{ij})$, in the sense of identifying simplices with simplices. After taking one barycentric subdivision, consider the cover
    \[ \{ \St(x): x \in X_0 \}. \]
    Every $x \in X_0$ lies in some $U_i$, so this is clearly a refinement of $(U_i)$. It covers $X$, because barycentric subdivision forces every simplex of $X$ to be adjacent to a vertex of $X$. By \cref{lem:StarsAreBoundaryConnected}, the sets $\St(x)$ are boundary-connected, and we only need to show that an overlap is boundary-connected as well. But we have
    \[ \St(x) \cap \St(y) = \St (\{x, y \}), \]
    so this likewise follows from \cref{lem:StarsAreBoundaryConnected}. \qedhere
\end{proof}

\begin{definition}
    Let $ \mcU = (U_i)$ be a boundary-tame cover. By \cref{prop:ctsext} and the cocycle condition, each $s_{ij}$ extends to a continuous homeomorphism
    \[ \ovl{s_{ij}} \colon \ovl{U_{ij}} \to \ovl{U_{ji}}. \]
    The \emph{intrinsic closure} is the space
    \[ \hat{X} = \coprod_i \ovl{U_i} / R, \]
    where $R$ identifies $x$ with $\ovl{s_{ij}}(x)$.
    The \emph{intrinsic boundary} is
    \[ \hat \partial X := \hat{X} \setminus X. \]
\end{definition}

\begin{figure}[ht]
  \centering
  \includegraphics[width=0.6\linewidth,
    alt={Three closed charts U i bar, U j bar, and U k bar drawn as overlapping disks, glued along their overlaps at three marked points. The outer arcs of the combined outline, drawn solid in dark red and labeled with the intrinsic boundary of X, contribute to the intrinsic boundary; the dashed arcs inside the overlaps are boundary points of individual charts that land in the interior of X.}]{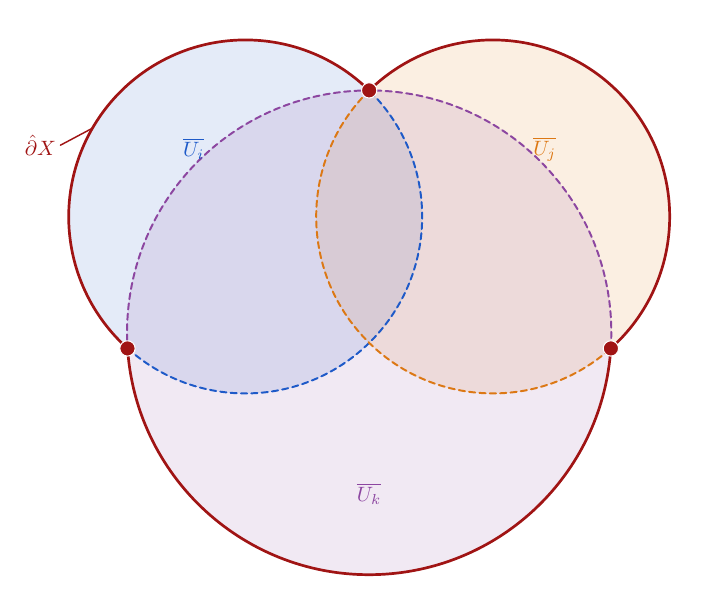}
  \caption{Not all boundary points of a given chart will land in the intrinsic boundary $\hat \partial X$.}
  \label{fig:intrinsic-closure}
\end{figure}

\begin{proposition}\label{prop:IntrinsicClosureProperties}
    \leavevmode\par
    \begin{enumerate}
        \item The canonical map $i \colon X \to \hat{X}$ is an open embedding.
        \item The space $\hat{X}$ is compact Hausdorff.
        \item The space $\hat{X}$ is independent of the cover $\mcU$.
        \item The subspace $X$ is boundary-connected.
        \item Intrinsic closure is functorial for maps between Riemann surfaces.
    \end{enumerate}
\end{proposition}

We will require the following lemma, which says that the closures $\ovl{s_{ij}}$ ``almost'' satisfy the cocycle condition.

\begin{lemma}\label{lem:IdentificationChain}
    Let $(U_i)$ be a boundary-tame cover, and suppose that $x \in U_i, y \in \ovl{U_j}$. If $x \sim y$ in $\hat{X}$, then we have $\ovl{s_{ij}}(x) = y$.
\end{lemma}

\begin{proof}
    We prove the following by induction: If
    \[ x \xr{s_{ik_1}} z_1 \to \cdots \to z_m  \]
    is a chain of identifications with $z_i \in \ovl{U_{k_i}}$, then we have $\ovl{s_{ik_m}}(x) = z_m$. Indeed, by \cref{lem:BoundaryInteriorIdentifications}, $U_{k_1 i}$ is a boundary neighborhood of $z_1$. Since $z_1$ is identified to $z_2$, we must have $z_1 \in \partial U_{k_1 k_2}$. It follows that $U_{k_1 i} \cap U_{k_1 k_2}$ is nonempty near $z_1$, and by the cocycle condition we must have
    \[ \ovl{s_{i k_2}}(x) = z_2. \]
    This shortens the chain and we conclude by induction. \qedhere
\end{proof}

\begin{proof}[Proof of \cref{prop:IntrinsicClosureProperties}]
    \leavevmode\par
    \textbf{(1)}: Suppose that $x \in U_i, y \in U_j$ are identified in $\hat{X}$. By \cref{lem:IdentificationChain}, we have $\ovl{s_{ij}}(x) = y$. But since $X$ is Hausdorff, this implies that $s_{ij}(x) = y$, so $x$ and $y$ are already identified in $X$.

    \textbf{(2)}: The relation $R: x \sim \ovl{s_{ij}}(x)$ is closed, so it suffices to show that its iteration stabilizes after finitely many steps. Thus, suppose we have a chain
    \[ z_1 \to z_2 \to \cdots \to z_m, \]
    where $z_k \in \ovl{U_{i_k}}$. If some $z_k$ lies in the interior $U_{i_k}$, then by \cref{lem:IdentificationChain}, we reduce to a chain $z_1 \to z_k \to z_m$ of length at most 3. Otherwise, we may assume that every $z_k$ lies along the boundary. If either $U_{i_k, i_{k-1}}$ or $U_{i_k, i_{k+1}}$ forms a boundary neighborhood of $z_k$, then they must meet near $z_k$. We thus have
    \[ \ovl{s_{i_{k-1}, i_{k+1}}}(z_{k-1}) = z_{k+1}, \]
    and we can shorten the chain by omitting $z_k$. But by \cref{prop:BoundaryAgreement}, for every $U_{ij}$ there are only finitely many boundary points where $U_{ij}$ is not a boundary neighborhood. Taking the union gives an explicit finite set $E \subset \coprod \ovl{U_i}$, and any chain of minimal length has $z_k \in E$ for $1 < k < m$. Removing duplicates, the chain has length at most $|E|+2$.

    \textbf{(3)}: By \cref{lem:BoundaryTameRefinements}, we reduce to showing the following: If $\mcV$ is a boundary-tame cover and $\mcU$ is a boundary-tame refinement, then the canonical map
    \[ \hat{X}^{\mcU} \to \hat{X}^{\mcV} \]
    is an isomorphism. It is clearly surjective, and both spaces are compact Hausdorff, so we only need to show injectivity. Choose open sets $U_i \subset V_j, U_{i'} \subset V_{j'}$: It is enough to show that if $x \in \ovl{U_i}, x' \in \ovl{U_{i'}}$ are identified via $\ovl{s_{j, j'}^{\mcV}}$, then they are identified in $\hat{X}^{\mcU}$. Since $\mcU$ forms a cover near $x$, there exists some $U_a$ with $U_a \cap V_{j, j'} \neq \emptyset$ near $x$. By the same token, there is some $U_b \subset V_{j'}$ which meets the image of $U_a$.

    Therefore, we reduce to the case when $j = j'$. If $U_i \cap U_{i'}$ is nonempty near $x$, this case is trivial. But since $V_j$ is connected near $x$, they must be linked by a chain of such overlaps and we conclude.

    \textbf{(4)}: If $X$ is not boundary-connected, then there exists a boundary point $y \in \hat \partial X$, a neighborhood $V \ni y$, and a decomposition
    \[ V \cap X = V_1 \cup \cdots \cup V_n, \]
    with $y$ corresponding to a point $y_i \in \partial V_i$. We can complete $(V_i)$ to a cover of $X$ by adjoining open sets whose closures avoid $y$ in $\hat{X}$. If $(U_i)$ is a boundary-tame refinement of $(V_i)$, then it defines the same intrinsic closure $\hat{X}$. But in this space, the points $y_i$ are not identified with each other, giving a contradiction.

    \textbf{(5)}: Let $f \colon X \to X'$ be a map of Riemann surfaces, let $(V_i)$ be a boundary-tame cover of $X'$, and let $(U_j)$ be a boundary-tame cover of $X$ refining $(f^{-1}(V_i))$. Then there are natural maps
    \[ U_j \to V_{i(j)}, \]
    and by \cref{prop:ctsext} they extend to closures. This is evidently compatible with transitions, so we obtain a natural map
    \[  \hat{f} \colon \hat{X} \to \hat{X'}. \]
    Since $X, X'$ are dense in their intrinsic closures, the map $\hat{f}$ is the unique continuous extension of $f$. From this it follows that $f \mapsto \hat{f}$ is functorial. \qedhere
\end{proof}

\subsection{Proof of the compactification theorem (Theorem~C)}

Now, assume that $S$ is unary-analytic and contains $\bbR_{\an}$. If $y \in \hat \partial X$, a \emph{boundary chart} for $y$ is a boundary neighborhood isomorphic to some open $U \subset \bbC$. The technical heart of \cref{thm:compactification} is the following:

\begin{proposition}\label{prop:EmbeddingConditions}
    The following are equivalent:
    \begin{enumerate}
        \item There exists an open embedding $X \inj Y$ into a compact Riemann surface, such that $\hat{X} \to Y$ is a closed embedding.
        \item There exists any open embedding $X \inj Y$.
        \item Every $y \in \hat \partial X$ has a boundary chart.
    \end{enumerate}
\end{proposition}

The implication $(1 \to 2)$ is trivial, and the implication $(2 \to 3)$ is relatively straightforward: Suppose that an embedding $f \colon X \inj Y$ exists. By functoriality, there is a corresponding map
\[ \hat{f} \colon \hat{X} \to \hat{Y}, \]
and manifestly $\hat{Y} = Y$, since $Y$ is compact. If $x \in \hat{X}$, and $\hat{f}(x) = y$, then let $U$ be a chart for $Y$ near $y$. The inverse image $(\hat{f})^{-1}(U)$ is a neighborhood of $x$, and since the original map $f$ is an open embedding, we have
\[ X \cap (\hat{f})^{-1}(U) \subset U. \]
This gives a boundary chart.

It remains only to prove $(3 \to 1)$. The idea of the proof is very simple: First, fix a cover of $X$ near the boundary $\hat \partial X$. Using results from \cref{sec:DefinableTopology}, pass to a cover whose transition functions are overconvergent, and then use the extended transition functions to glue charts past the boundary. The choice of neighborhoods, however, requires some care.

For the rest of the section, keep the standing assumption that every $y \in \hat \partial X$ has a boundary chart.

\begin{figure}[ht]
  \centering
  \includegraphics[width=\linewidth,
    alt={Left: a surface X with one handle, ending in a short cylinder, the boundary collar C. The two rims of the cylinder, the inner one in blue and the outer one in red, together form the intrinsic boundary of C; the outer red rim is labeled with the intrinsic boundary of X. Right: a surface X prime containing the same cylinder C in its interior, with its red rim labeled with the intrinsic boundary of X, now lying inside X prime.}]{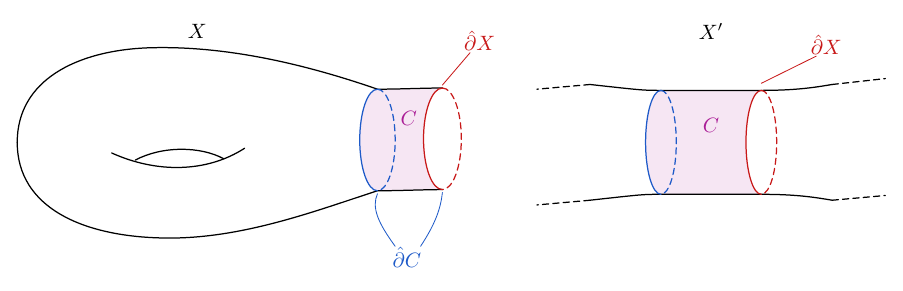}
  \caption{The idea of \cref{lem:CollarEmbeddingSuffices}.}
  \label{fig:boundary-collar}
\end{figure}

\begin{lemma}\label{lem:CollarEmbeddingSuffices}
    Let $C \subset X$ be a boundary collar.
    \begin{enumerate}
        \item The boundary $\hat \partial X$ is naturally identified with a clopen subset of $\hat \partial C$.
        \item Suppose that there exists a Riemann surface $X'$, and an embedding $C \inj X'$ with the following property: $\hat \partial X \subset \hat{C}$ is mapped into $X'$, and the map is a closed embedding. Then condition (1) of \cref{prop:EmbeddingConditions} follows.
    \end{enumerate}
\end{lemma}

\begin{proof}
    \textbf{(1)}: By functoriality, there is a natural map $\hat{i} \colon \hat{C} \to \hat{X}$, and we let
    \[ D = \hat{i}^{-1}(\hat \partial X). \]
    It follows that $D$ is closed. Meanwhile, since $C$ is a boundary collar, there exists an open neighborhood $C'$ of $\hat \partial X$ with $C' \cap X = C$, and we have
    \[ \hat{i}^{-1}(C') = C \cup D. \]
    It follows that $D$ is relatively open in $\hat \partial C$, and therefore clopen. Furthermore, the map $D \to \hat \partial X$ is an isomorphism, because $C$ is a boundary collar.

    \textbf{(2)}: Let $f \colon C \to X'$ be the embedding, and choose a neighborhood $U'$ for the image of $\partial X$ in $X'$. If we let $U = U' \cap C$, then by part (1) we can restrict and assume that
    \[  \ovl{U} \cap \hat \partial C \subset \hat \partial X. \]
    Define a space $X''$ with charts $X, U'$, glued in the natural way along $U$. By the condition on $\ovl{U}$, $X''$ is Hausdorff, and therefore a Riemann surface. We thus obtain a definable embedding $X \inj X''$ with the following properties:
    \begin{enumerate}[label=(\alph*)]
        \item Under functoriality, $\hat{X}$ lands in $X''$.
        \item The map $\hat \partial X \to X''$ is a closed embedding.
    \end{enumerate}
    The Riemann surface $X''$ has finite type, so by Stout's theorem there exists an embedding $X'' \inj Y$ into a compact Riemann surface. By condition (a), $X$ is relatively compact in $X''$; since $S$ contains $\bbR_{\an}$, the composition
    \[ X \inj X'' \inj Y \]
    is $S$-definable. Condition (b) shows that the map $\hat{X} \to Y$ is injective, and therefore a closed embedding. \qedhere
\end{proof}

\begin{remark}[Notation]
    For the rest of the section, let $(A_i)$ denote a boundary cover of $X$ by charts. Let $(A_i')$ be some open cover of $\hat \partial X$ with
    \[ A_i' \cap X = A_i, \]
    and let $A_i^{\partial} = A_i' \cap \hat \partial X$ be the corresponding open cover of the boundary itself. We adopt the corresponding notation for other boundary covers as well.
\end{remark}

\begin{proposition}\label{prop:BoundaryCoverhasTameRefinement}
    After refinement, we can assume that the $\ovl{A_i'}$ have no triple intersections, and that $(A_i)$ is boundary-tame.
\end{proposition}

\begin{proof}
    Since the $A_i$ form a boundary cover, every $x \in \hat \partial X$ is contained in some $A_i^\partial$. Thus, the cover
    \[ \{ V_x := \St_X (x): x \in (\hat \partial X)_0  \} \]
    is a refinement of $(A_i)$. Subdividing first if necessary, we may assume that there is no vertex of $X$ adjacent to three vertices of $\hat \partial X$, and this ensures that the cover has no triple intersections; refining slightly, we can assume the closures likewise have no triple intersections. Since $X$ is boundary-connected, one deduces from \cref{lem:SimplicialCharacterizationofBoundaryConnectedness} that the sets $\St_X(x)$ are boundary-connected as well. If $\St_X(x)$ meets $\St_X(x')$, then we have $\{x, x' \} \in (\hat \partial X)_1$, and
    \[ \St_X(x) \cap \St_X(x') = \St_X(\{x, x' \}), \]
    so the same argument shows the double intersections are boundary-connected as well. \qedhere
\end{proof}

Let $s_{ij} \colon A_{ij} \to A_{ji}$ be the transition functions. Since the $A_{ij}$ are boundary-connected, it follows that each $s_{ij}$ is overconvergent except at finitely many boundary points. By \cref{lem:FiniteOmissionsLemma}, there exists a refined boundary cover $(B_i)$ making the $s_{ij}$ overconvergent. The sets $(B_i')$ have no triple intersections, and taking a further refinement as in \cref{prop:BoundaryCoverhasTameRefinement}, we can also assume that $(B_i)$ is boundary-tame.

Let $\Omega_{ij} \supset \ovl{B_{ij}}$ be a domain of overconvergence, and let $\tilde s_{ij}$ be the extension of the transition function. Then $\tilde s_{ij}(\Omega_{ij})$ is open and contains $\ovl{B_{ji}}$. This shows, along with the cocycle condition, that we can choose the domains $\Omega_{ij}$ to match, and require the $\tilde s_{ij}$ to be mutually inverse biholomorphisms.

\begin{lemma}
    If $x \in \ovl{B_{ij}} \cap \hat \partial X$, then there is a neighborhood $U_x$ of $x$ on which we have
    \[ \Gamma(\tilde s_{ij}) \cap (\ovl{B_i} \times \ovl{B_j}) \subset \Gamma(\ovl{s_{ij}}). \]
    In other words, all identifications are from $\ovl{B_{ij}}$ to $\ovl{B_{ji}}$.
\end{lemma}

\begin{proof}
    Since $(B_i)$ is a boundary cover with no triple intersections, we have either $x \in B_i^\partial$ or $y \in B_j^{\partial}$, with $y = \ovl{s_{ij}}(x)$. If we are in the second case, it suffices to produce a corresponding neighborhood $V_y \ni y$ and then pull it back. Thus, we may assume without loss of generality that $x \in B_i^\partial$.

    Since the cover is boundary-tame, $B_i$, $B_{ij}$ are connected at $x$, and $B_j, B_{ji}$ are connected at $y$. Furthermore, since $\tilde s_{ij}$ is a biholomorphism near $x$, the set
    \[  \beta_j := (\tilde s_{ij})^{-1}(B_j) \]
    is connected at $x$. It follows that $B_i, B_{ij}$ and $\beta_j$ can all be expressed as radial intervals near $x$, and we have $B_{ij} \subset B_i \cap \beta_j$. Since $X$ is Hausdorff, we have
    \[ \partial B_{ij} \subset \partial B_i \cup (\ovl{s_{ij}})^{-1}( \partial B_j), \]
    and we note that
    \[ (\ovl{s_{ij}})^{-1}(\partial B_j) \cap \partial B_{ij} = (\tilde s_{ij})^{-1}(\partial B_j) \cap \partial B_{ij}. \]
    It follows that $B_{ij}$ is one component of the intersection $B_i \cap \beta_j$. If $\beta_j$ is strictly larger than $B_{ij}$, then since $\beta_j$ is connected at $x$, it must contain one of the boundary arcs of $B_{ij}$, which is therefore a boundary arc of $\partial B_i$. But this contradicts the assumption that $B_i$ is a boundary chart. The desired result now follows by taking closures. \qedhere
\end{proof}

\begin{corollary}
    There exist open neighborhoods $\Theta_{ij}$ of $\ovl{B_{ij}} \cap \hat \partial X$ with the following property:
    \[  (\ovl{\Theta_{i,j}} \times \bbC) \cap \Gamma(\tilde s_{ij}) \cap (\ovl{B_i} \times \ovl{B_j}) \subset \Gamma(\ovl{s_{ij}}). \]
    We can require that $\Theta_{ij} \subset \Omega_{ij}$, and that they are compatibly identified by $\tilde s_{ij}$.
\end{corollary}

\begin{proof}
    Take $\Theta_{ij}$ at first to be the union of the neighborhoods $U_x$ from the previous lemma. Then the desired intersection property holds for $\Theta_{i,j} \times \bbC$, and refining we can assume it holds on the closure as well. By a further refinement, we can require the final two properties to hold as well. \qedhere
\end{proof}

\begin{lemma}
    There exist open sets $V_i \supset B_i^\partial$ which satisfy the following intersection property for $i \neq j$:
    \[ V_i \cap (\tilde{s}_{ij})^{-1}(V_j) \Subset \Theta_{ij}. \]
    We can also choose the $V_i$ to have no triple intersections.
\end{lemma}

\begin{proof}
    For all $i \neq j$, choose open sets $ \Psi_{i,j}$ with
    \[ \ovl{B_{ij}} \cap \hat \partial X \subset \Psi_{ij} \Subset \Theta_{ij} \]
    and define the closed set
    \[ \Gamma_{i,j}  = \{(x, \tilde s_{ij}(x)): x \in \ovl{\Theta_{ij}} \setminus \Psi_{i, j} \}. \]
    Since the $\ovl{B_i'}$ have no triple intersections, it follows that $\ovl{B_i^\partial} \cap \ovl{B_j^\partial} \subset \ovl{B_{ij}}$. Indeed, this is obvious if $x \in B_i^\partial$ or $x \in B_j^\partial$. But since $(B_i')$ has no triple intersections, any $x$ that meets both $\ovl{B_i^\partial}$ and $\ovl{B_j^\partial}$ must be contained in the interior of one of them. Therefore, we have
    \[ \ovl{B_i^\partial} \cap \ovl{B_j^\partial} \subset \ovl{B_{ij}} \cap \hat \partial X. \]
    It follows by the lemma we just proved that the compact product $\ovl{B_i^\partial} \times \ovl{B_j^\partial}$ is disjoint from $\Gamma_{i,j}$. Thus, by the tube lemma there exist open neighborhoods $V_i, V_j$ of $\ovl{B_i^\partial}, \ovl{B_j^\partial}$ with $V_i \times V_j$ disjoint from $\Gamma_{i,j}$. Concretely, this means that
    \[ V_i \cap (\tilde s_{ij})^{-1}(V_j) \subset \Psi_{i,j}. \]
    Repeating finitely many times for each $i \neq j$ and taking a common refinement with no triple intersections, we conclude. \qedhere
\end{proof}

Define the space $X'$ with charts $(V_i)$ and overlaps
\[ V_{ij} := V_i \cap \tilde{s}_{ij}^{-1}(V_j). \]
To fix notation, let $t_{i,j} \colon V_{i,j} \to V_{j, i}$ be the transition function, which is just the restriction of $\tilde{s}_{ij}$. The cocycle condition is automatic since there are no triple intersections. By construction, we have
\[ \Gamma(t_{ij}) = \Gamma(\tilde{s}_{ij}) \cap (V_i \times V_j). \]
Since $\ovl{V_{ij}} \subset \Omega_{ij}$, continuity of $\tilde s_{ij}$ implies that $\Gamma(t_{ij})$ is closed in $V_i \times V_j$, and it follows that $X'$ is Hausdorff.

Thus, $X'$ is a Riemann surface. If we let $U_i = V_i \cap X$, and let $C = \cup_i U_i$, then $C$ is a boundary collar, and the inclusions $U_i \inj V_i$ give a relatively compact embedding $f \colon C \to X'$.

It remains to show that the map $\hat{f} \colon \hat \partial X \to X'$ is injective, equivalently an embedding. By construction the transition maps $\tilde s_{ij}$ reduce along $\hat \partial X$ to the original maps $\ovl{s_{ij}}$, so this is true by the definition of the intrinsic closure. This completes the proof of \cref{prop:EmbeddingConditions}. \qed

\begin{proof}[Proof of \cref{thm:compactification}]
    Let $S$ be a branching structure containing $\bbR_{\an}$, and let $X$ be a definable Riemann surface with no definable punctures. Recall what this means: The complement $\hat X \setminus X$ has no isolated points. By \cref{prop:EmbeddingConditions}, it remains only to show, for every $y \in \hat \partial X$, that there exists a boundary chart near $y$. To show this, we give an argument analogous to the one we gave for \cref{prop:H1Vanishing}.

    Fix any cover of $X$. By \cref{lem:BoundaryTameRefinements}, it has a refinement $(U_i)$ where the $U_i$ are boundary-connected charts. Pick some $y \in \hat \partial X$: We will show by descending induction on the number of cover elements that there exists a single chart near $y$. At each stage we have a cover by boundary-connected charts, and we reduce the number of charts by one.

    Choose some $i$ with $y \in \ovl{U_i}$: If $U_i$ is not a boundary chart, then by boundary-connectedness of $X$ we have $y \in \ovl{U_{ij}}$ for some $j$. Restrict near $y$: By \cref{lem:ConnectednessAtZero}, if we identify $y$ with $0$ then we can write $U_i = (f, g)$ as a radial interval. The overlap $U_{ij}$ is a finite union of radial intervals, and applying \cref{lem:ClosedExcision}, we can refine and eliminate all intervals, except those meeting the boundary arcs $f, g$. If $U_{ij} = U_i$, then we can eliminate $U_i$ from the cover near $y$.

    Thus, we can write $U_{ij}$ as a union of at most two proper subintervals, each one meeting $f$ or $g$. Suppose first that
    \[ U_{ij} = (f, h_1) \cup (h_2, g) \]
    has both intervals as endpoints. If we express $U_j$ in a similar fashion, then by \cref{lem:BoundaryInteriorIdentifications}, the same conclusion holds in this chart: The overlap $U_{ji}$ contains both boundary arcs of $U_j$. It follows that $(U_i, U_j)$ forms a cover near $y$. All nearby points in $\partial U_i$ are $j$-interior, and vice versa, so $y \in \hat \partial X$ is isolated, violating our assumption on $X$.

    Finally, we reduce to the case where $U_{ij}$ is a single radial interval. Suppose it has the form
    \[ U_{ij} = (f, h),  \]
    with $f < h <g$. Since $S$ is branching, it follows after shrinking near $0$ that the transition function $s_{ij}$ has a definable holomorphic extension $\tilde{s}_{ij}$ to $U_i$. Restricting both sets near zero, we can then define a holomorphic function
    \[ t \colon U_i \cup U_j \to \bbC, \]
    by $t = \id$ on $U_j$ and $t = \tilde{s}_{ij}$ on $U_i$. By o-minimality, the function $t$ has bounded argument, and replacing $t$ by $t(z)^{1 / n}$, we can make its argument as small as desired. By \cref{lem:UnivalenceCriterion}, the restriction $t|_{U_i}$ is univalent, with image a radial interval $(a,b)$. Suppose without loss of generality that $t(U_{ij})$ is the radial interval $(c,b)$: Applying \cref{lem:UnivalenceCriterion} again, $t|_{U_j}$ is univalent, and its image must be a radial interval $(c,d)$. It follows that $t$ as a whole is univalent, a biholomorphism onto $(a,d)$. Therefore, we can take $V := U_i \cup U_j$ to be a single chart near $y$. Since $U_i \cup U_j$ is connected near $y$, we conclude. \qedhere
\end{proof}

\begin{proof}[Proof of \cref{cor:TameEmbeddability}]
    Let $Y$ be a compact Riemann surface, and let $U \subset Y$ be a definable open subset. Let $p_1, \ldots, p_n$ be the punctures of $U$, and let
    \[ V = U \cup \{p_i \}. \]
    By \cref{thm:compactification}, it suffices to show that $V$ has no definable punctures. Since intrinsic compactifications are functorial, the inclusion extends to a continuous map $\hat{i} \colon \hat{V} \to Y$. Given any $v \in \hat \partial V$, we obtain a boundary chart near $v$ by taking a neighborhood $W$ of $y:= \hat{i}(v)$ and intersecting with $V$. The point $y$ is not isolated in $\partial V$, so there exists a curve germ $C \subset \partial V$ limiting to $y$. Applying \cref{lem:ConnectednessAtZero}, take a boundary-tame cover of $W \cap V$ by radial intervals $(W_i)$, and assume that $W_1$ has $C$ as a boundary arc. Refining, we may assume that $C \subset \partial W_1$ is not in the closure of any $W_{1i}$. By \cref{prop:IntrinsicClosureProperties}, $(\ovl{W_i})$ represents $\hat{V}$, and it follows that the map $C \to \hat \partial{V}$ is injective. It follows that $v$ is not isolated. \qedhere
\end{proof}

\printbibliography

\end{document}